\documentclass[a4paper,12pt]{article}
\usepackage{amsmath}
\usepackage{amssymb}
\usepackage{setspace}
\usepackage{fullpage}
\usepackage{bbm}
\usepackage{bm}
\usepackage{amsthm}
\newtheorem{theorem}{Theorem}[section]
\newtheorem{lemma}[theorem]{Lemma}

\newtheorem{corollary}[theorem]{Corollary}

\theoremstyle{definition}
\usepackage{color}

\newtheorem{definition}{Definition}

\newtheorem*{remark}{Remark}

\usepackage{pdfsync}
\def\PG{\mathrm{PG}}  
\def\SL{\mathrm{SL}}\def\PSL{\mathrm{PSL}}

\def\Aut{\mathrm{Aut}}
\def\PGammaL{\mathrm{P}\Gamma\mathrm{L}}
\def\PGL{\mathrm{PGL}} \def\PSL{\mathrm{PSL}}
\def\GL{\mathrm{GL}}

 \def\B{\mathcal{B}} 
\def\D{\mathcal{D}} \def\E{\mathcal{E}}

\def\l{\lambda}
\def\la{\langle}
\def\ra{\rangle}

\def\L{\mathcal{L}}  
 \def\P{\mathbb{P}}
 \def\S{\mathcal{S}}
\def\Ff{\mathcal{F}}
\def\Fb{\overline{\mathcal{F}}}
\def\Pqt{\mathbb{P}_{q^t}^{r}}
\def\Pq{\mathbb{P}_q^{rt}}

\def\GammaL{\Gamma\mathrm{L}}

\def\Fqtr{\mathbb{F}_{q^t}^r}
\def\Fqtrs{\mathbb{F}_{q^t}^{r*}}

\def\F{\mathbb{F}}

\title{Desarguesian spreads and field reduction for elements of the semilinear group}
\author{Geertrui Van de Voorde \thanks{This author is a postdoctoral fellow of the Research Foundation Flanders (FWO -- Vlaanderen).}}
\begin{document}
\date{}
\maketitle
\begin{abstract} The goal of this note is to create a sound framework for the interplay between field reduction for finite projective spaces, the general semilinear groups acting on the defining vector spaces and the projective semilinear groups. This approach makes it possible to reprove a result of Dye on the stabiliser in $\PGL$ of a Desarguesian spread in a more elementary way, and extend it to $\PGammaL(n,q)$. Moreover a result of Drudge \cite{Drudge} relating Singer cycles with Desarguesian spreads, as well as a result on subspreads (by Sheekey, Rottey and Van de Voorde \cite{ABB}) are reproven in a similar elementary way. Finally, we try to use this approach to shed a light on Condition (A) of Csajb\'ok and Zanella, introduced in the study of linear sets \cite{corrado}.
\end{abstract}


{\bf Mathematics Subject Classfication:} 51E20, 51E23, 05B25,05E20
\section{Introduction}
In the last two decades a technique, commonly referred to as `field reduction', has been used in many constructions and characterisations in finite geometry, e.g. the construction of $m$-systems and semi-partial geometries (SPG-reguli) \cite{deirdre,Thas}, eggs and translation generalised quadrangles \cite{FGQ}, spreads and linear blocking sets \cite{guglielmo}, $\ldots$. A survey paper of Polverino \cite{olga} collects these and other applications of linear sets in finite geometry. For projective and polar spaces, the technique itself was more elaborately discussed in \cite{FQ11}. The main goal of this paper is to contribute to this study by also considering the interplay between finite projective spaces and the projective (semi-)linear groups acting on the definining spaces. A lot of results in this paper are well-known or considered to be folklore. However, as it is our aim to provide a sound framework, we chose to reprove them here. In this way, we hope to provide a self-contained paper collecting several results which are related to field reduction and Desarguesian spreads, in particular, we reprove (and extend to $\PGammaL$) a theorem of Dye on the stabiliser of a Desarguesian spread in Theorem \ref{Dye}, we integrate the study of Singer cycles in this framework, giving an explicit proof for a result of Drudge in Theorem \ref{drud}. In Theorem \ref{uniek}, we give a new direct proof for the fact that a Desarguesian spread has a unique Desarguesian subspread, avoiding the use of indicator sets.

Finally, in Section \ref{cor} we try to give some more information on linear sets satisfying the Condition (A) introduced by Csajb\'ok and Zanella. In the appendix (Section \ref{appendix}), we give some more details about the embedding of $\PGL(r,q^t)$ in $\PGL(rt,q)$, motivated by a misconception that has appeared in the literature.

\section{Field reduction for vector spaces and general semilinear groups}
\subsection{Definitions and the Desarguesian vector space partition}
It is well-known that the finite field $\F_{q^t}$ forms a $t$-dimensional vector space over $\F_q$, and hence, that an $r$-dimensional vector space $V$ over $\F_{q^t}$ corresponds to an $rt$-dimensional vector space $W$ over $\F_q$. To this end, the vectors of $W$ are identified with the vectors of $V$, addition in $W$ is defined as the addition in $V$, but we only allow scalar multiplication with elements of $\F_q$. It is easy to see that  in this way $W$ forms an $\F_q$-vector space, and since it has $q^{rt}$ vectors, it has dimension $rt$ over $q$.
In this construction, the vectors in $W$ are identified with the vectors of $V$ but to be able to define everything what follows in a precise way, we introduce a map $\Ff$ which maps the vector $v$ in $V$ to its corresponding vector $\Ff(v)$ in $W$. We let $\Ff: \Fqtr\to \F_q^{rt}$, $v\mapsto \Ff(v)$ denote this bijection between vectors, i.e.
$$\Ff(v)=\Ff(w)\iff v=w,$$ and see that $\Ff$ is an $\F_q$-linear map, i.e.
\begin{eqnarray*}
\Ff(v+w)&=&\Ff(v)+\Ff(w)\\
\Ff(\l v)&=&\l \Ff(v) \ \mathrm{ for } \ \l \in \F_q.
\end{eqnarray*}
We call $\Ff$ a {\em vector field reduction map}, and abbreviate this as {\em a vfr-map}.

The map $\Ff$ as defined here is not unique, there are multiple choices of such a bijection, corresponding to different choices of a basis of $\F_{q^t}$ over $\F_q$. More explicitely, we can construct a vector field reduction map as follows: fix a basis $\{v_1,\ldots,v_t\}$ of $\mathbb{F}_{q^t}$ over $\mathbb{F}_q$ and define $\Ff:\F_{q^t}^r\to\F_q^{rt}$, $(w_1,\ldots,w_r)\mapsto (\alpha_{11},\alpha_{12},\ldots,\alpha_{1t},\alpha_{21},\ldots,\alpha_{2t},\ldots,\alpha_{r1},\ldots,\alpha_{rt})$ with $w_i=\sum_{j=1}^t \alpha_{ij}v_j$.

In the remainder of this paper, we will always fix the vector field reduction map. The following lemma justifies this, since the images under different vector field reduction maps are equivalent under the general linear group.
\begin{lemma} Let $\Ff: \Fqtr\to \F_q^{rt}$, and $\Ff': \Fqtr\to \F_q^{rt}$ be vector field reduction maps. If we define $\xi(\Ff(v))=\Ff'(v)$ for all $v$ of $\F_{q^t}^r$, then $\xi\in \GL(rt,q)$. 
\end{lemma}
\begin{proof} Let $w_1,w_2\in \F_q^{rt}$, then $w_1=\Ff(v_1), w_2=\Ff(v_2)$ for some $v_1,v_2 \in \F_{q^t}^r$. Now $\xi(w_1+w_2)=\xi(\Ff(v_1)+\Ff(v_2))=\xi(\Ff(v_1+v_2))=\Ff'(v_1+v_2)=\Ff'(v_1)+\Ff'(v_2)=\xi(\Ff(v_1))+\xi(\Ff(v_2))=\xi(w_1)+\xi(w_2)$. For $\lambda \in \F_q$, we have that $\xi(\lambda w_1)=\xi(\lambda \Ff(v_1))=\xi(\Ff(\lambda v_1))=\Ff'(\l v_1)=\l\Ff'(v_1)=\l\xi(\F(v_1))=\l\xi(w_1)$, hence $\xi\in \GL(rt,q)$.
\end{proof}


A {\em vector space partition $\S$} of a vector space $V$ is a set of subspaces such that every non-zero vector of $V$ is contained in a unique member of $\S$.

\begin{lemma} \label{partition} Let $\Ff: \Fqtr\to \F_q^{rt}$ be a  a vfr-map. Fix a vector $v\in \Fqtrs$, then $\{\Ff(\alpha v)\mid \alpha \in \F_{q^t}\}$ defines a $t$-dimensional $\F_q$-vector space of $\F_q^{rt}$. Moreover, the set $\{ \{\Ff(\alpha v)\mid \alpha \in \F_{q^t}\}\mid v\in \F_{q^t}^{r*} \}$ defines a vector space partition of $\F_q^{rt}$ into $q^r$ $t$-dimensional $\F_q$-vector spaces.
\end{lemma}
\begin{proof} We see that $\Ff(\alpha v)+\l \Ff(\alpha' v)$, with $\l$ in $\F_q$, $\alpha,\alpha'\in \F_{q^t}$ and $v\in \F_{q^t}^{r*}$ equals $\Ff(\alpha v)+\Ff(\l\alpha'v)=\Ff((\alpha+\l\alpha')v)$, hence, that $\{\Ff(\alpha v)\mid \alpha\in \F_{q^t}\}$ forms an $\F_q$-subspace $S$. The number of vectors in $S$ equals $q^t$, hence, $S$ is a $t$-dimensional vector space over $\F_q$.

If for some $v,w\in \F_{q^t}^r$, $\{\Ff(\alpha v)\mid \alpha\in \F_{q^t}\}\cap \{\Ff(\alpha' w)\mid \alpha'\in \F_{q^t}\}\neq 0$, then, since $\alpha_1 v=\alpha'_1 w$ for some $\alpha_1,\alpha'_1\in \F_{q^t}$, it follows that $v=\frac{\alpha'_1}{\alpha_1}w$  and that $\{\Ff(\alpha v)\mid \alpha\in \F_{q^t}\}= \{\Ff(\alpha' w)\mid \alpha'\in \F_{q^t}\}$, which proves the statement.
\end{proof}


For reasons that will become clear later, we call the vector space partition found in the previous lemma the {\em Desarguesian vector space partition of $\F_q^{rt}$, obtained from $\Ff: \Fqtr\to \F_q^{rt}$}.

\subsection{Field reduction for semilinear maps}
Let $\tau$ be an element of $\GammaL(r,q^t)$, then $\tau$ is a semilinear map from $\F_{q^t}^r$ to itself, i.e.
\begin{align*}
\tau:  \F_{q^t}^r & \rightarrow \F_{q^t}^r\\
\tau(v+w)&=\tau(v)+\tau(w)\\
\tau(\alpha v)&=\psi(\alpha)\tau(v), \ {\mathrm{for\ some}}\ \psi \in \Aut(\F_{q^t}),
\end{align*}
and $\tau$ is invertible. Note that $\tau$ determines the element $\psi\in \Aut(\F_{q^t})$ in a unique way; we say that $\psi$ is the field automorphism {\em associated with $\tau$}. The set $\Aut(\F_{q^t})$ is the set of all field automorphisms of $\F_{q^t}$ and it is well-known that every field automorphism of $\F_{q^t}$ is of the form $x\rightarrow x^{p^i}$ for some $i$, where $q$ is a power of the prime $p$. Note that for a field automorphism $\psi$ it holds that $\psi(\alpha \beta)=\psi(\alpha)\psi(\beta)$. It easily follows that $\psi$ fixes every subfield of $\F_{q^t}$, since the elements in a subfield $\F_{q^{t'}}$ are precisely those elements $x\in \F_{q^t}$ for which $x^{q^{t'}}=x$. Multiplication of two elements $\tau_1$ and $\tau_2$ in $\GammaL(r,q^t)$ is defined as the composition $\circ$ of maps and we use the notation $\tau_2\tau_1$ for $\tau_2\circ \tau_1$.

\begin{lemma}\label{lem10}Let $B=\{v_1,\ldots,v_r\}$ be a basis of $\Fqtr$, and let $\xi$ be the following map, where $\mu_i\in \F_{q^t}$, $i=1,\ldots,r$.
\begin{eqnarray*}\xi:\F_{q^t}^r&\rightarrow& \F_{q^t}^r\\
\sum_{i=1}^r \mu_i v_i&\mapsto& \sum_{i=1}^r \psi(\mu_i)\xi'(v_i), \ {\mathrm{for\ some\ fixed }}\ \psi\in \Aut(\F_{q^t})\ \mathrm{and} \ \xi'\in \GL(r,q^t).\end{eqnarray*} Then $\xi$ defines an element of $\GammaL(r,q^t)$ and every element of $\GammaL(r,q^t)$ is of this form. With respect to a fixed basis $B$, the element $\xi$ determines $\xi'\in \GL(r,q^t)$ in a unique way and we may write $\xi=(\xi',\psi)_B$. 
\end{lemma}

\begin{proof} It is straightforward to check that $\xi$ defines a semilinear map and, using that $\xi'$ and $\psi$ are invertible, that $\xi$ is invertible. Now let $\tau$ be an invertible semilinear map with $\tau(\alpha v)=\psi(\alpha)\tau(v)$. Define $\xi'$ to be the $\F_{q^t}$-linear map which maps $v_i$ to $\tau(v_i)$, i.e. $\xi'(\sum_{i=1}^r \mu_i v_i)=\sum_{i=1}^r\mu_i \tau(v_i)$. Then $\tau(\sum_{i=1}^r \mu_i v_i)=\sum_{i=1}^r \tau(\mu_i v_i)=\sum_{i=1}^r \psi(\mu_i)\tau(v_i)=\sum_{i=1}^r \psi(\mu_i)\xi'(v_i)$ with $\xi' \in \GL(r,q^t)$ since $\tau$ is invertible. Hence, $\tau$ is of the prescribed form.

We have already seen before that $\psi$ is the unique field automorphism associated with $\xi$. If for a fixed basis $B=\{v_1,\ldots,v_r\}$ the equation $\sum_{i=1}^r \psi(\mu_i)\xi_1'(v_i)$ $=\sum_{i=1}^r \psi(\mu_i)\xi_2'(v_i)$ holds for all choices $\mu_i$ in $\F_{q^t}$, then it easily follows that $\xi_1'(v_k)=\xi_2'(v_k)$ for all $k=1,\ldots,r$, and hence, that $\xi_1'=\xi_2'$ since they are $\F_{q^t}$-linear maps coinciding on the basis $B$. Hence if $(\xi'_1,\psi)_B=(\xi'_2,\psi)_B$, then $\xi_1'=\xi_2'$, which proves the last part of the statement.
\end{proof}

We are now ready to define an action of the map $\Ff$ on the elements of $\GammaL(r,q^t)$.

\begin{definition} \label{def}Let $\xi$ be an element of $\GammaL(r,q^t)$. Define the action of $\Ff$ on an element of $\GammaL(r,q^t)$ as follows.
\begin{align*}
\Ff(\xi): \F_{q}^{rt}&\rightarrow \F_{q}^{rt}\\
\Ff(v)&\mapsto \Ff(\xi)(\Ff(v)):=\Ff(\xi v).
\end{align*}
\end{definition}

In other words, we define $\Ff(\xi)$ as the unique mapping which lets the following diagram commute.
\[
\begin{array}{clccc}
\Fqtr&&\xrightarrow{\xi}&&\Fqtr\\
\downarrow&\Ff& &\Ff&\downarrow\\
\F_{q}^{rt}&&\xrightarrow{\Ff(\xi)}&&\F_q^{rt}
\end{array}\]

\begin{theorem} \label{embedding}Let $\Ff: \Fqtr\to \F_q^{rt}$ be a vector field reduction map, and consider the action of $\Ff$ on elements of $\GammaL(r,q^t)$ as defined in Definition \ref{def}.
The map $\Ff$ defines an embedding of $\GammaL(r,q^t)$ in $\GammaL(rt,q)$, and an embedding of $\GL(r,q^t)$ in $\GL(rt,q)$.
\end{theorem}
\begin{proof} Let $\xi$ be an element of $\GammaL(r,q^t)$ with associated field automorphism $\psi$. Then we first show that $\Ff(\xi)$ is an element of $\GammaL(rt,q)$. We have that 
$\Ff(\xi)(\Ff(v)+\Ff(v'))=\Ff(\xi(v+v'))=\Ff(\xi(v)+\xi(v'))=\Ff(\xi(v))+\Ff(\xi(v'))=\Ff(\xi)(\Ff(v))+\Ff(\xi)(\Ff(v'))$ and hence, $\Ff(\xi)$ is additive. Moreover, $\Ff(\xi)(\l \Ff(v))$ with $\l\in \F_q$, equals $\Ff(\xi)(\Ff(\lambda v))=\Ff(\xi(\l v))=\Ff(\psi(\l)\xi(v))$. Now, as $\l \in \F_q$, $\psi(\l)\in \F_q$, so $$\Ff(\psi(\l)\xi(v))=\psi(\l)\Ff(\xi(v))$$
and
$$\Ff(\xi)(\l \Ff(v))=\psi(\l)\Ff(\xi)(\Ff(v)).$$ This implies that $\Ff(\xi)$ is indeed an element of $\GammaL(rt,q)$. 
If $\Ff(\xi_1)=\Ff(\xi_2)$, then $\Ff(\xi_1)(\Ff(v))=\Ff(\xi_2)(\Ff(v))$ for all $v \in \Fqtr$, and hence, $\Ff(\xi_1(v))=\Ff(\xi_2(v))$ for all $v \in \Fqtr$. This implies that $\xi_1$ and $\xi_2$ coincide on $\Fqtr$, and that $\Ff$ defines an injection of $\GammaL(r,q^t)$ in $\GammaL(rt,q)$.

Finally, we see that $\Ff(\xi_2\xi_1)(\Ff(v))=\Ff((\xi_2 \xi_1)(v))=\Ff(\xi_2(\xi_1(v)))=\Ff(\xi_2)(\Ff(\xi_1(v)))=\Ff(\xi_2)(\Ff(\xi_1)(\Ff(v)))=(\Ff(\xi_2) \Ff(\xi_1))(\Ff(v)).$ This shows that $\Ff$ is compatible with the multiplication in $\GammaL(r,q^t)$, which is the composition of maps, and hence, $\Ff$ is an embedding. The last part follows by observing that, if $\xi$ is an element of $\GL(r,q^t)$, then $\psi=id$, and $\Ff(\xi)=\Ff(\xi')$ is an element of $\GL(rt,q)$.
\end{proof}

\subsection{The stabiliser of the Desarguesian vector space partition}
In this subsection, we determine the stabiliser of the Desarguesian vector space partition in $\F_q^{rt}$ obtained from $\Ff: \Fqtr\to \F_q^{rt}$. We start with the elementwise stabiliser, i.e. those group elements which leave every vector subspace of the vector space partition invariant.
\begin{theorem}\label{vectorelementwise} Let $\Ff: \Fqtr\to \F_q^{rt}$ be a vfr-map. Suppose that $\phi$ is an element of $\GammaL(rt,q)$ stabilising the Desarguesian vector space partition of $\F_q^{rt}$ obtained from $\Ff$ elementwise. Then $\phi=\Ff(m_\beta)$, for some $\beta\in \F_{q^t}^*$, where $m_\beta$ is the element of $\GL(r,q^t)$ representing the multiplication with $\beta$, i.e. mapping a vector $v$ onto  the vector $\beta v$.
\end{theorem}
\begin{proof} As $\phi$ stabilises the Desarguesian vector space partition elementwise, we have that for all $v\in \F_{q^t}^{r*}$
$$\{ \phi(\Ff(\alpha v))\mid \alpha \in \F_{q^t}\}=\{\Ff(\alpha'v)\mid \alpha' \in \F_{q^t}\}.$$
For each $v\in \F_{q^t}^{r*}$, define $\psi_v(\alpha)$ as follows:
$$\phi(\Ff(\alpha v))=\Ff(\psi_v(\alpha) v).$$

Consider now the equality $\phi(\Ff((\alpha+\alpha')v))=\Ff(\psi_v(\alpha+\alpha')v)$. The left hand side of this equation equals $\phi\Ff(\alpha v)+\phi\Ff(\alpha' v)=\Ff((\psi_v(\alpha)+\psi_v(\alpha'))v)$, which yields that $\psi_v(\alpha+\alpha')=\psi_v(\alpha)+\psi_v(\alpha')$. 
Now let $v$ and $w$ be linearly independent in $\Fqtr$ and consider $\phi(\Ff(\alpha(v+w)))$. We obtain that $\psi_{v+w}(\alpha)v+\psi_{v+w}(\alpha)w=\psi_v(\alpha)v+\psi_w(\alpha)w$. Since $v$ and $w$ are linearly independent, this shows that $\psi_{v+w}(\alpha)=\psi_v(\alpha)=\psi_w(\alpha)$. As $\beta v$ and $w$, $\beta \in \F_{q^t}$ are linearly independent over $\F_{q^t}$, this argument also shows that $\psi_{\beta v}(\alpha)=\psi_w(\alpha)=\psi_v(\alpha)$ and we find that $\psi_v(\alpha)=\psi_w(\alpha)$ for all $v,w\in \F_{q^t}^{r*}$. Put $\psi=\psi_v$.

Consider the equality $\phi(\Ff(\beta \alpha v))=\Ff(\psi_{\alpha v}(\beta) \alpha v)=\Ff(\psi_{v}(\beta \alpha) v)$. Since $\psi_{\alpha v}=\psi_v=\psi$, $\psi (\beta\alpha)=\alpha\psi(\beta)$ for all $\alpha, \beta\in \F_q^t$. In particular, $\psi(\alpha)=\psi(1)\alpha$ for all $\alpha$.

We see that $\phi(\Ff(v)=\Ff(\psi(1)v)$ for all $v\in \Fqtr$. So if we define $\xi$ to be the mapping $m_{\psi(1)}$, then indeed $\Ff(\xi)=\phi$ since $\Ff(\xi)(\Ff(v))=\Ff(\xi(v))=\Ff(\psi(1)v)=\phi(\Ff(v))$ for all $v\in \F_{q^t}^r$ which concludes the proof.
\end{proof}

\begin{corollary} The elementwise stabiliser in $\GammaL(rt,q)$ of the Desarguesian vector space partition of $\F_q^{rt}$, obtained from $\Ff: \Fqtr\to \F_q^{rt}$,  is contained in $\GL(rt,q)$ and is isomorphic to $\F_{q^t}^\ast$ (the multiplicative group of $\F_{q^t}$). 
\end{corollary}
\begin{proof} An element $\phi$ in $\GammaL(rt,q)$ stabilising the Desarguesian vector space partition is of the form $\Ff(m_\beta)$, and hence, is an $\F_q$-linear map, i.e. $\phi\in \GL(rt,q)$. As every $\beta\in \F_{q^t}^\ast$ gives rise to a unique element $\Ff(m_\beta)$ that clearly stabilises the Desarguesian vector space partition elementswise, the statement follows.
\end{proof}

We now turn our attention to the setwise stabiliser of the Desarguesian vector space partition.

\begin{lemma} \label{lem2}Let $\phi$ be an element of $\GammaL(rt,q)$ stabilising the Desarguesian vector space partition of $\F_q^{rt}$ obtained from the field reduction map $\Ff: \Fqtr\to \F_q^{rt}$. Let, for $v\in \F_{q^t}^{r*}$ and $\alpha\in \F_{q^t}$, $\psi_v(\alpha)$ be the element satisfying $$\phi(\Ff(\alpha v))=\Ff(\psi_v(\alpha)\Ff^{-1}\phi\Ff(v)).$$ Then 
$\psi_v=\psi_w$ for all $v,w\in \F_{q^t}^{r*}$ and $\psi:=\psi_v$ is an element of $\Aut(\F_{q^t})$. If $\phi$ is in $\GL(rt,q)$, then $\psi$ fixes each element of $\F_q$.
\end{lemma}
\begin{proof}
We will prove this lemma in several steps. 
\begin{enumerate}
\item {\bfseries\boldmath$ \psi_v(\alpha)$ is well defined.} Fix an element $v\neq 0$ in $\F_{q^t}^r$. It is clear that $\psi_v(1)=1$ since $\phi(\Ff(v))=\Ff(1.\Ff^{-1}\phi\Ff(v))$. Put $z_v:=\Ff^{-1}\phi\Ff(v)$, then we see that $\phi\Ff(v)=\Ff(z_v)$. As $\phi$ stabilises the Desarguesian vector space partition, we obtain that there is a unique element $\psi_v(\alpha)$ satisfying
$$\phi(\Ff(\alpha v))=\Ff(\psi_v(\alpha)z_v)=\Ff(\psi_v(\alpha)\Ff^{-1}\phi\Ff(v)).$$

\item {\bfseries\boldmath$\psi_v$ is additive.} We see that $\phi(\Ff(\alpha v)+\Ff(\alpha' v))=\phi(\Ff((\alpha+\alpha')v)=\Ff(\psi_v(\alpha+\alpha')z_v)$. On the other hand, $\phi(\Ff(\alpha v)+\Ff(\alpha' v))=\phi(\Ff(\alpha v))+\phi(\Ff(\alpha' v))=\Ff(\psi_v(\alpha)z_v)+\Ff(\psi_v(\alpha')z_v)=\Ff((\psi_v(\alpha)+\psi_v(\alpha'))z_v)$. This implies that $\psi_v(\alpha+\alpha')=\psi_v(\alpha)+\psi_v(\alpha')$.
\item {\bfseries\boldmath$\psi_v(\alpha \beta)=\psi_{\beta v}(\alpha)\psi_v(\beta)$
.} We see that $\phi(\Ff(\alpha \beta v))=$ $\Ff(\psi_{\beta v}(\alpha)\Ff^{-1}\phi\Ff(\beta v))=$ $\Ff(\psi_{\beta v}(\alpha)\Ff^{-1}\Ff\psi_v(\beta)\Ff^{-1}\phi\Ff(v))=\Ff(\psi_{\beta v}(\alpha)\psi_v(\beta)\Ff^{-1}\phi\Ff(v))$. On the other hand, $\phi(\Ff(\alpha \beta v))=\Ff(\psi_v(\alpha \beta)\Ff^{-1}\phi\Ff(v))$. Hence, $\psi_v(\alpha \beta)=\psi_{\beta v}(\alpha)\psi_v(\beta)$ for all $\alpha,\beta\in \F_{q^t}$.
\item {\bfseries\boldmath$\psi_v$ is surjective.} Let $\beta\in \F_{q^t}$. Since $\phi$ is an element of $\GammaL(rt,q)$, $\phi$ is surjective and hence, there is an element $w\in \F_q^{rt}$ such that $\phi(w)=\Ff(\beta\Ff^{-1}\phi\Ff(v))$. Now $\phi$ is an element of the stabiliser of the Desarguesian vector space partition so if we put $z=\Ff^{-1}\phi\Ff(v)$, then, since $\phi\Ff(v)=\Ff(z)$, $\phi(\Ff(\alpha v))=\Ff(\beta z)$ for some $\alpha\in \F_{q^t}$. Hence, $w=\Ff(\alpha v)$ and $\psi_v(\alpha)=\beta$.
\item {\bfseries\boldmath$\psi_v=\psi_w$ for all $v,w$ in $\F_{q^t}^{r*}$}. Let $v$ and $w$ be linearly independent vectors in $\Fqtr$. Consider the equality 
$$\phi(\Ff(\alpha(v+w)))=\phi(\Ff(\alpha v))+\phi(\Ff(\alpha w)).$$ The left hand side equals $\Ff(\psi_{v+w}(\alpha)\Ff^{-1}\phi\Ff(v)+\psi_{v+w}(\alpha)\Ff^{-1}\phi\Ff(w))$ while the right hand side equals $\Ff(\psi_v(\alpha)\Ff^{-1}\phi\Ff(v)+\psi_w(\alpha)\Ff^{-1}\phi\Ff(w))$. This implies that 
\begin{align}
(\psi_{v+w}(\alpha)-\psi_v(\alpha))\Ff^{-1}\phi\Ff(v)=(\psi_w(\alpha)-\psi_{v+w}(\alpha))\Ff^{-1}\phi\Ff(w).\label{eq1}
\end{align}
Now the vectors $\Ff^{-1}\phi\Ff(v)$ and $\Ff^{-1}\phi\Ff(w)$ are linearly independent: suppose to the contrary that $\Ff^{-1}\phi\Ff(w)=\beta\Ff^{-1}\phi\Ff(v)$ for some $\beta\in \F_{q^t}$, then by (4), there is an $\alpha\in \F_{q^t}$ such that $\beta=\psi_v(\alpha)$, and hence, $\Ff^{-1}\phi\Ff(w)=\psi_v(\alpha)\Ff^{-1}\phi\Ff(v)=\Ff^{-1}\phi\Ff(\alpha v)$ from which it follows that $w=\alpha v$, a contradiction. By Equation (\ref{eq1}), this implies that $\psi_v(\alpha)=\psi_w(\alpha)$ for all linearly independent vectors, $v,w$. Since $\gamma v$ and $w$ are also linearly independent, we find that $\psi_{\gamma v}=\psi_w=\psi_v$, and hence $\psi_v$ is the same map for all non-zero vectors $v$ of $\Fqtr$. 
\item {\bfseries\boldmath$\psi=\psi_v$ is an element of $\Aut (\F_{q^t})$.} Put $\psi=\psi_v$, then we see from (3), namely $\psi_v(\alpha \beta)=\psi_{\beta v}(\alpha)\psi_v(\beta)$ that $\psi(\alpha \beta)=\psi(\alpha)\psi(\beta)$, so $\psi_v$ is multiplicative. As $\psi$ is additive by (2), $\psi$ is an element of $\Aut(\F_{q^t})$.
\item Finally, we will show that {\bfseries\boldmath $\psi(\l)=\l$ for $\l\in \F_q$ if $\phi\in \GL(rt,q)$}. We have on one hand that $\phi(\Ff(\l v))=\phi(\l\Ff(v))=\l \phi\Ff(v)$ since $\phi$ in $\GL(rt,q)$ and on the other hand, $\phi(\Ff(\l v))=\Ff(\psi(\l)\Ff^{-1}\phi\Ff(v))$. By letting $\Ff^{-1}$ act on both sides, we obtain that $\Ff^{-1}(\l\phi\Ff(v))=\psi(\l)\Ff^{-1}\phi\Ff(v)$. Since $\Ff$ is an $\F_q$-linear map, $\Ff^{-1}$ is an $\F_q$-linear map, so $\Ff^{-1}(\l\phi\Ff(v))=\l\Ff^{-1}(\phi\Ff(v))$. This implies that $\psi(\l)=\l$.
\qedhere
\end{enumerate}
\end{proof}

\begin{lemma}\label{omgekeerd} Let $\Ff: \Fqtr\to \F_q^{rt}$ be a vfr-map. Let $\phi=\Ff(\xi)$ for some $\xi$ in $\GammaL(r,q^t)$, then $\phi$ stabilises the Desarguesian vector space partition of $\F_q^{rt}$ obtained from $\Ff$.
\end{lemma}
\begin{proof} Suppose that $\psi$ is the field autormorphism associated with $\xi$. Since $\phi=\Ff(\xi)$ we have that $\Ff(\xi)(\Ff(\alpha v))=\Ff(\xi(\alpha v))=\Ff(\psi(\alpha)\xi(v))$ with $v\in \Fqtr$. This implies that $\{\Ff(\xi)(\Ff(\alpha v))\mid \alpha\in \F_{q^t}\}=\{\Ff(\beta w)\mid\beta\in \F_{q^t}\}$ for some $w$ in $\Fqtr$.
\end{proof}

\begin{theorem}\label{stabiliservector} Let $\Ff: \Fqtr\to \F_q^{rt}$ be a vfr-map. Let $\phi$ be an element of $\GammaL(rt,q)$ stabilising the Desarguesian vector space partition of $\F_q^{rt}$ obtained from $\Ff$, then 
$\phi=\Ff(\xi)$ for some $\xi$ in $\GammaL(r,q^t)$. If $\phi$ is an element of $\GL(rt,q)$ then the field automorphism $\psi\in \Aut(\F_{q^t})$ associated with $\xi$ fixes every element of $\F_q$.
\end{theorem}
\begin{proof} Using the definitions of Lemma \ref{lem2}, we obtain that for all $v \in \Fqtr$, $\phi(\Ff(\alpha v))=\Ff(\psi(\alpha)\Ff^{-1}\phi\Ff(v))$. Put 
\begin{eqnarray*}
\xi:\Fqtr&\rightarrow&\Fqtr\\
\alpha v&\mapsto& \psi(\alpha)\Ff^{-1}\phi\Ff(v).
\end{eqnarray*}
The map $\xi$ is well-defined: suppose that $\alpha v=\alpha' w$ are non-zero vectors. Since $w=\frac{\alpha}{\alpha'}v$, $\Ff^{-1}\phi\Ff(w)=\Ff^{-1}\phi\Ff(\frac{\alpha}{\alpha'}v)=\Ff^{-1}\Ff(\psi(\frac{\alpha}{\alpha'})\Ff^{-1}\psi\Ff(v)$. Since $\psi\in \Aut(\F_{q^t})$, we have that $\psi(\frac{\alpha}{\alpha'})=\frac{\psi(\alpha)}{\psi(\alpha')}$ and hence $\xi(\alpha' w)=\psi(\alpha')\Ff^{-1}\phi\Ff(w)=\psi(\alpha)\Ff^{-1}\phi\Ff(v)=\xi(\alpha v)$.

Moreover, $\xi(v+v')=\Ff^{-1}\phi\Ff(v+v')=\xi(v)+\xi(v')$ and $\xi(\alpha v)=\psi_(\alpha)\xi(v)$ with $\psi$ in $\Aut(\F_{q^t})$ which shows that $\xi\in \GammaL(r,q^t)$. Finally, since for all $v\in \Fqtr$
$$\Ff(\xi)(\Ff(v))=\Ff(\xi(v))=\Ff(\Ff^{-1}\phi\Ff(v))=\phi(\Ff(v)),$$
we have that $\Ff(\xi)=\phi$. Finally, by Lemma \ref{lem2} (7), $\psi$ fixes every element of $\F_q$ if $\phi\in \GL(r,q^t)$.
\end{proof}

The subgroup of elements of $\Aut(\F_{q^t})$ that fix every element of $\F_q$, is denoted by $\Aut(\F_{q^t}/\F_q)$, this subgroup is generated by the map $x\mapsto x^q$.
The following theorem was for $\GL(rt,q)$ shown by Dye in \cite{Dye}, using different methods. 

\begin{theorem}  Let $\Ff: \Fqtr\to \F_q^{rt}$ be a vfr-map. The stabiliser in $\GammaL(rt,q)$ of the Desarguesian vector space partition in  $\F_q^{rt}$ obtained from $\Ff$ is isomorphic to $\GammaL(r,q^t)\cong \GL(r,q^t)\rtimes \Aut(\F_{q^t})$. 
The stabiliser in $\GL(rt,q)$of the Desarguesian vector space partition obtained from $\Ff$ is isomorphic to $\GL(r,q^t)\rtimes \Aut(\F_{q^t}/\F_q)$.
\end{theorem}
\begin{proof} Using Lemma \ref{omgekeerd} and Theorem \ref{stabiliservector}, we see that the stabiliser of the Desarguesian vector space partition in $\GammaL(rt,q)$ is isomorphic to $\GammaL(r,q^t)$ and the stabiliser of the Desarguesian vector space partition in $\GL(rt,q)$ consists of all elements $\Ff(\xi)$ where $\xi$ has associated field automorphism $\psi\in \Aut(\F_{q^t}/\F_q)$. Recall that $G\cong N\rtimes H$ if and only if $N$ and $H$ are subgroups of $G$ and there exists a homomorphism $G\to H$ which is the identity on $H$.

We can see $\Aut(\F_{q^t})$  as a subgroup of $\GammaL(r,q^t)$  by mapping an element $\psi$ of $\Aut(\F_{q^t})$ to the semilinear map $\xi_\psi$ defined as $\xi_\psi:\sum_{i=1}^r \mu_i v_i \mapsto \sum_{i=1}^r \psi(\mu_i)v_i$, for a basis $B=\{v_1,\ldots,v_r\}$ of $\Fqtr$, or in other words, by considering the semilinear map $(id,\psi)_B.$


Consider the map $f$ whichs maps an invertible semilinear map $\xi$ onto its associated field automorphism $\psi$. The map $f$ has as kernel exactly all invertible linear maps, hence, the elements of $\GL(r,q^t)$. It is clear that $f(\GammaL(r,q^t))=\Aut(\F_{q^t})$, and $f$ is the identity on $\Aut(\F_{q^t})$. Hence, we obtain the well-known fact that $\GammaL(r,q^t)\cong \GL(r,q^t)\rtimes \Aut(\F_q)$. Now the image under $f$ of the stabiliser of the Desarguesian vector space partition in $\GL(rt,q)$ is precisely the group $\Aut(\F_{q^t}/\F_q)$, which proves the theorem.
\end{proof}

\section{The field reduction map for projective spaces and projective general semilinear groups}

\subsection{Projective semilinear maps and Desarguesian spreads}

Let $\Pqt=\PG(r-1,q^t)$ be the projective space corresponding to $\Fqtr$, i.e. every point of $\Pqt$ corresponds to a set of vectors $\la v\ra_{q^t}:=\{\alpha v\mid\alpha \in \F_{q^t}\}$, where $v\in (\Fqtr)^*$. Likewise every subspace of $\Pqt$ (sometimes called an $\F_{q^t}$-subspace for clarity) is a set $\la U\ra_{q^t}:=\{\la v\ra_{q^t}\mid v \in U^*\}$, where $U$ is a vector subspace of $\F_{q^t}^r$. Note that if $U=0$, then $\la U\ra_q$ is the empty projective subspace.

Let $\Pq=\PG(rt-1,q)$ denote the projective space corresponding to $\F_q^{rt}$, i.e. every point of $\Pq$ corresponds to a set of vectors $\la v\ra_q:=\{\l v\mid\l \in \F_q\}$ where $v\neq 0$. Likewise every subspace of $\Pq$, sometimes denoted an $\F_{q}$-subspace for clarity, is a set $\{\la v\ra_{q}\mid v \in U^*\}$, where $U$ is a vector subspace of $\F_q^{rt}$.

Consider the group of invertible semilinear maps $\GammaL(r,q^t)$. The {\em centre} $Z$ of $\GammaL(r,q^t)$ consists of all invertible scalar maps, i.e. $Z=\{m_\beta\mid \beta \in \F_{q^t}^*\}$, where as before $m_\beta$ denotes the map
\begin{eqnarray*}
m_\beta:\Fqtr&\rightarrow& \Fqtr\\
v&\mapsto& \beta v.
\end{eqnarray*}
Consider an element of $\PGammaL(r,q^t)\cong \GammaL(r,q^t)/Z$, then this corresponds to the set $\{m_\alpha \xi \mid \alpha \in \F_{q^t}^*\}$ for some $\xi\in \GammaL(r,q^t)$. 

In analogy with the notation for projective points, we denote this set as $\la \xi\ra_{q^t}$. Similarly, an element of $\PGammaL(rt,q)$ is denoted as $\la \phi \ra_q$ for some $\phi\in \GammaL(rt,q)$.

The following lemma is of course well-known.
\begin{lemma} $\PGammaL(r,q^t)\cong \PGL(r,q^t)\rtimes \Aut(\F_{q^t})$. 
\end{lemma}
\begin{proof} Pick a basis $B$ for $\F_{q^t}^r$. With every element $\psi$ of $\Aut(\F_{q^t})$, we associate the element $\langle \rho_\psi\rangle_{q^t}$ of $\PGammaL(r,q^t)$, where $\rho_\psi$ is the element $(id,\psi)_B$ of $\GammaL(r,q^t)$. This association is a bijection: if there exist $\alpha,\beta\in \F_{q^t}$ such that $m_\alpha \rho_{\psi_1} (v)=m_\beta\rho_{\psi_2}(v)$ for all $v$, then it follows that $\alpha\psi_1(\mu_i)=\beta\psi_2(\mu_i)$ for all $\mu_i\in \F_{q^t}$, which yields by considering $\mu_i=1$ that $\alpha=\beta$ and hence that $\psi_1=\psi_2$. This shows that $\Aut(\F_{q^t})$ is a subgroup of $\PGammaL(r,q^t)$. Now consider the mapping $\langle (\xi,\psi)_B\rangle_{q^t}\mapsto \langle\rho_\psi\rangle_{q^t}$, where $(\xi,\psi)_B$ is as before the semilinear map defined by $\xi\in \GL(r,q^t)$ and $\psi\in \Aut(\F_{q^t})$ (see Lemma \ref{lem10}). We find that this is a well-defined homomorphism with image $\Aut(\F_{q^t})$ and kernel all the elements of the form $\langle \xi\rangle_{q^t}$, $\xi\in \GL(r,q^t)$, hence, $\PGL(r,q^t)$.
\end{proof}

From Theorem \ref{embedding}, we obtain an embedding of $\GammaL(r,q^t)$ in $\GammaL(rt,q)$. Also in the case of projective spaces it is possible to define a field reduction map, which will now be denoted as $\Fb$. We will prove results, similar to Theorem \ref{stabiliservector}. However, we will not have a well-defined extension of $\Fb$ to elements of $\PGammaL(r,q^t)$.

\begin{lemma} \label{Bpunt} Let $\la v\ra_{q^t} \in \Pqt$ and consider the set $\Fb(\la v\ra_{q^t})=\la \{ \Ff(\alpha v) \mid \alpha \in \F_{q^t}\}\ra_q$. Then $\Fb(\la v\ra_{q^t})$ is a $(t-1)$-dimensional projective $\F_q$-subspace of $\Pq$ and $\Fb(\la v\ra_{q^t})=\{\la \Ff(\alpha v)\ra_q \mid \alpha \in \F_{q^t}^\ast\}$.
\end{lemma}
\begin{proof} From Lemma \ref{partition}, we have that $U=\{\Ff(\alpha v)\mid \alpha \in \F_{q^t}\}$ is a $t$-dimensional $\F_q$-subspace $U$ of $\F_q^{rt}$, hence, $\la U\ra_q$ is a $(t-1)$-dimensional projective subspace of $\Pq$. Moreover, $\la U\ra_q=\{ \la w\ra_{q}\mid w\in U^*\}=\{\la w\ra_q\mid w\in \{\Ff(\alpha v)\mid \alpha \in \F_{q^t}^*\}\}=\la \{ \Ff(\alpha v) \mid \alpha \in \F_{q^t}\}\ra_q=\Fb(\la v\ra_{q^t})$.
\end{proof}

Define the following map, for $v\in \F_{q^t}^r*$:
\begin{eqnarray*}
\Fb:\Pqt&\rightarrow& \Pq\\
\la v\ra_{q^t}&\mapsto& \{\la \Ff(\alpha v)\ra_q \mid \alpha \in \F_{q^t}^\ast\}.
\end{eqnarray*}
This map $\Fb$ is the {\em field reduction map} from $\Pqt$ to $\Pq$, as defined in \cite{FQ11}. Note that the vector field reduction map maps a vector onto a vector whereas the field reduction map maps a projective point onto a projective subspace.

\begin{remark}\label{opmerkinginbedding} 
It is well-known that an element of $\PGammaL(r,q^t)$ acts in a natural way on the points of a projective space: if $\la \xi\ra_{q^t}$ is an element of $\PGammaL(r,q^t)$ and $\la v\ra_{q^t}$ is a point of $\Pqt$, then $\la \xi\ra_{q^t}(\la v\ra_{q^t})$ is defined as $\la \xi(v)\ra_{q^t}$. This action is clearly well-defined.

It seems a natural idea to extend the definition of $\Fb$ to elements of $\PGammaL(r,q^t)$ by defining $\Fb(\la \xi\ra_{q^t})$ as $\la \Ff(\xi)\ra_q$. However, this mapping is not well-defined as $\la \xi \ra_{q^t}=\la \alpha \xi\ra_{q^t}$ but it is not hard too check that $\la \Ff(\alpha \xi)\ra_q$ does not equal $\la \Ff(\xi)\ra_q$ if $\alpha\notin \F_q$. We give some more information on the embedding of $\PGammaL(r,q^t)$ in $\PGammaL(rt,q)$ in Section \ref{appendix}.
\end{remark}

A {\em $(t-1)$-spread} of a projective space $\Pq$ is a partition of the points of $\Pq$ into $(t-1)$-dimensional subspaces.

\begin{lemma}\label{lem6} Consider the set $\D=\{\Fb(\la v\ra_{q^t})\mid v\in \F_{q^t}^{r*}\}$, then $\D$ is a $(t-1)$-spread of $\Pq$.
\end{lemma}
\begin{proof} We have shown in Lemma \ref{partition} that $\{\{\Ff(\alpha v) \mid \alpha \in \F_{q^t}\} \mid v\in \Fqtrs\}$ is a vector space partition of $\F_q^{rt}$. It follows from Lemma \ref{Bpunt} that for every $v\in \Fqtrs$, $\Fb(\la v\ra_{q^t})$ is a $(t-1)$-dimensional projective subspace of $\Pq$. So we obtain that that $\D$ is a $(t-1)$-spread of $\Pq$.
\end{proof}

We call the set $\D$ in Lemma \ref{lem6}, the {\em spread obtained from the field reduction map $\Ff: \Fqtr\to \F_q^{rt}$}. A $(t-1)$-spread $\mathcal{S}$ in $\Pq$, $r>2$ is called {\em Desarguesian} if the incidence structure $(\mathcal{P},\mathcal{L}, \mathcal{I})$ with as point set $\mathcal{P}$ the elements of $\mathcal{S}$ in $\Pqt$, as line set $\L$ the subspaces spanned by two different elements of $\mathcal{S}$, and where incidence is containment, is isomorphic to the point-line incidence structure of the Desarguesian projective space $\Pqt$. It follows from this definition that a Desarguesian spread $\mathcal{S}$ in $\Pqt$ is {\em normal} (or {\em geometric}), i.e. it has the property that a subspace spanned by elements of $\mathcal{S}$ is partitioned by elements of $\mathcal{S}$.
A $(t-1)$-spread $\mathcal{S}$ in $\mathbb{P}_q^{2t}$, is called {\em Desarguesian} if the projective completion of the incidence structure obtained from the Andr\'e/Bruck-Bose construction starting from $\S$ is isomorphic to $\mathbb{P}(\F_{q^t}^2)=\mathbb{P}_{q^t}^3$.

\begin{lemma} Let $\D$ be the $(t-1)$-spread obtained from the field reduction map $\Ff: \Fqtr\to \F_q^{rt}$ in Lemma \ref{lem6}. Then $\D$ is a Desarguesian spread.
\end{lemma}
\begin{proof}First suppose that $r>2$. In order to show that $\D$ is a Desarguesian spread, we consider the incidence structure $(\mathcal{P},\mathcal{L},\mathcal{I})$, where $\mathcal{P}$ is the set of elements of $\D$, i.e. the sets $\{\la \Ff(\alpha v)\ra_q\mid \alpha \in \F_{q^t}^*\}$, $v\in \F_{q^t}^{r*}\}$ and the lines $\mathcal{L}$ are the sets spanned by elements of $\D$, i.e. the sets $\{\la \lambda \Ff(\alpha v)+\mu \Ff(\beta w)\ra_q \mid \alpha,\beta\in \F_{q^t}, (\l,\mu)\in (\F_q\times \F_q)\setminus (0,0)\}$ and $\mathcal{I}$ is containment.

It is clear that the set $\mathcal{L}$ consists of the sets of the form $\{ \la \Ff(\alpha'v)+\Ff(\beta'w)\ra_q\mid \alpha',\beta'\in \F_{q^t}\}=\{\la \Ff(\alpha' v+\beta'w)\ra_q\mid \alpha',\beta'\in \F_{q^t}\}$. Since $\Pqt$ has as points the elements $\la \alpha v\mid \alpha \in \F_{q^t}^*\}$, $v\in \F_{q^t}^{r*}$ and as lines the sets $\la \alpha'v+\beta'w\ra_{q^t}$, we see that $\Ff$ provides and isomorphism between $(\mathcal{P},\mathcal{L},\mathcal{I})$ and the point-line incidence structure of $\Pqt$.

If $r=2$, the elements of $\D$ can be considered as the points of a projective line that is contained in $\mathbb{P}_{q^t}^2$.
To see this, first notice that $\mathbb{P}_{q^t}^3$ can be modeled as arising from the vector space $\F_{q^t}^2\times \F_{q^t}$, i.e. the points are the sets $\la (v_1,v_2)\ra_{q^t}$, where $v_1\in \F_{q^t}^2$ and $v_2\in \F_{q^t}$. Likewise, the lines are the sets $\{\la \alpha(v_1,v_2)+\beta(v_1',v_2')\ra_{q^t}\mid\alpha, \beta\in \F_{q^t}\}$, with $(v_1',v_2')\neq \gamma (v_1,v_2)$ where $\gamma\in \F_{q^t}$. Now consider a vfr-map $\Ff_1$ from $\F_{q^t}^2$ to $\F_q^{2t}$ and a a vfr-map $\Ff_2$ from $\F_{q^t}$ to $\F_q$. Then it is clear that the mapping $f:\la (v_1,v_2)\ra_{q^t}\mapsto \{\la \Ff_1(\alpha v_1),\Ff_2(\alpha v_2)\ra_q\mid\alpha\in \F_{q^t}\}$ defines a mapping defined on the points of $\mathbb{P}_{q^t}^3$. The mapping $f$ maps a line $\la \alpha(v_1,v_2)+\alpha'(v_1',v_2')\ra_{q^t}$ onto $\{\la (\l \Ff_1(\beta_1v_1)+\l'\Ff_1(\beta'_1v_1'),\l \Ff_2(\beta_2v_2)+\l'\Ff_2(\beta'_2v_2')\ra_q\mid \beta,\beta'\in \F_{q^t}, (\l,\l')\in (\F_q\times \F_q)\setminus \{(0,0)\}\}$. It is not too hard too check that $f$ preserves incidence and maps the line of $\PG(2,q^t)$ corresponding to the vectors $(v_1,0)$ onto the a set which is isomorphic to the set of elements of $\D$.
\end{proof}

We have seen that the spread $\D$ obtained from the field reduction map $\Ff$ in Lemma \ref{lem6} is a Desarguesian spread; we will see in Corollary \ref{desequiv2} that all Desarguesian $(t-1)$-spreads in $\PG(rt-1,q)$ are $\PGL(rt,q)$-equivalent.


\subsection{The stabiliser of a Desarguesian spread}

We will now investigate the stabiliser of a Desarguesian $(t-1)$-spread in $\Pq$. We will first show that w.l.o.g., we may consider the spread $\D$ obtained from the field reduction map $\Ff: \Fqtr\to \F_q^{rt}$ in Lemma \ref{lem6}. We start with an easy lemma.

\begin{lemma} \label{omgekeerd2} Let $\D$ be the Desarguesian $(t-1)$-spread obtained in Lemma \ref{lem6} from the field reduction map $\Ff: \Fqtr\to \F_q^{rt}$ and let $m_\beta$ be the element of $\GammaL(r,q^t)$ mapping a vector $v$ onto  the vector $\beta v$, then $\la \Ff(m_\beta) \ra_q$ stabilises  $\D$ elementwise. Let $\xi$ in $\GammaL(r,q^t)$, then $\la \Ff(\xi)\ra_q$ stabilises $\D$.
\end{lemma}
\begin{proof} We use that $\langle \Ff(m_\beta)\rangle_q\la\Ff(\alpha v)\ra_q=\la \Ff(m_\beta(\alpha v))\ra_q=\la \Ff(\beta \alpha v)\ra_q$ from which the first part follows. Similarly, let $\psi$ be the field automorphism associated with $\xi$, then we have $\langle \Ff(\xi)\rangle_q\la\Ff(\alpha v)\ra_q=\la \Ff(\xi(\alpha v))\ra_q=\la \Ff(\psi(\alpha)\xi( v))\ra_q$.
\end{proof}

The following theorem is well-known and will be of use later. For a proof for $r=2$, see e.g. \cite[Theorem 25.6.7]{Hirschfeld-Thas} in which $\PGL$-equivalence is shown, for a proof for $r\geq 2$, see e.g. \cite[Corollary 16]{bader}. 
\begin{theorem}\label{desequiv} All Desarguesian $(t-1)$-spreads in $\Pq$ are $\PGammaL(rt,q)$-equivalent. 
\end{theorem}

Up to our knowledge, the statement that all Desarguesian $(t-1)$-spreads are $\PGL$-equivalent is not in the literature for $r>2$. However, this statement follows from $\PGammaL$-equivalence, and the following lemma.

\begin{lemma} Let $\mathcal{S}$ be a $(t-1)$-spread of $\Pq$ and let $\D$ be the Desarguesian $(t-1)$-spread obtained in Lemma \ref{lem6} from the field reduction map $\Ff: \Fqtr\to \F_q^{rt}$. If there exists an element $\la \phi\ra_{q}$ of $\PGammaL(rt,q)$ such that $\la\phi\ra_q(\mathcal{S})=\D$, then there exists an element $\la \phi'\ra_q$ of $\PGL(rt,q)$ such that $\la \phi'\ra_q(\mathcal{S})=\D$.
\end{lemma}
\begin{proof} Let $\phi\in \GammaL(rt,q)$ with $\la \phi\ra_q(\mathcal{S})=\D$ and suppose that $\psi$ is the field automorphism in $\Aut(\F_q)$, associated with $\phi$. Since $\psi\in \Aut(\F_q)$, $\psi$ is a mapping which maps every $x\in \F_q$ onto $x^{p^k}$, for some fixed integer $k$. Consider the automorphism $\tilde{\psi}$ of $\Aut(\F_{q^t})$, which maps every element $y \in \F_{q^t}$ onto $y^{p^k}$, then it is clear that $\tilde{\psi}(x)=\psi(x)$ for $x\in \F_q$, and hence, $\tilde{\psi}^{-1}(x)=\psi^{-1}(x)$ for $x\in \F_q$.
Pick a basis $B$ for $\F_{q^t}^r$ and consider the element $\la \Ff(\xi_{\tilde{\psi}^-1})\ra_q$ of $\PGammaL(r,q^t)$, where $\xi_{\tilde{\psi}^{-1}}$ is the element of $\GammaL(r,q^t)$ defined as $\xi_{\tilde{\psi}^{-1}}=(id,\tilde{\psi}^{-1})_B$ as in Lemma \ref{lem10}. 

Consider now the element $\la \Ff(\xi_{\tilde{\psi}^-1})\ra_q \la\phi\ra_q$ $=\la \Ff(\xi_{\tilde{\psi}^-1}) \phi\ra_q$ of $\PGammaL(rt,q)$. Then $(\Ff(\xi_{\tilde{\psi}^-1}) \phi) (\lambda v)$, where $\lambda \in \F_q$ and $v\in \F_{q}^{rt}$, equals $\Ff(\xi_{\tilde{\psi}^-1}) \psi(\lambda)\phi(v)$ $=\tilde{\psi}^{-1}(\psi(\lambda))\Ff(\xi_{\tilde{\psi}^-1})\phi(v)$. Since $\lambda\in \F_q$, $\psi(\lambda)\in \F_q$ and $\tilde{\psi}^{-1}(\psi(\lambda))=\psi^{-1}(\psi(\lambda))=\lambda$, and hence, $\Ff(\xi_{\tilde{\psi}^-1}) \phi (\lambda v)=\lambda \Ff(\xi_{\tilde{\psi}^-1})\phi (v)$. This implies that the field automorphism associated with $\la \Ff(\xi_{\tilde{\psi}^-1}) \phi\ra_q$ is the trivial one, and hence, $\la \Ff(\xi_{\tilde{\psi}^-1}) \phi\ra_q$ is an element $\la \phi'\ra_q$ of $\PGL(rt,q)$. Moreover, since by Lemma \ref{omgekeerd2}, $\la \Ff(\xi_{\tilde{\psi}^-1})\ra_q$ stabilises $\D$, we have that  $\la \phi'\ra_q(\mathcal{S})=\la \Ff(\xi_{\tilde{\psi}^-1}) \phi\ra_q(\mathcal{S})=\la \Ff(\xi_{\tilde{\psi}^-1})\ra_q\la \phi\ra_q(\mathcal{S})=\la \Ff(\xi_{\tilde{\psi}^-1})\ra_q(\D)=\D$.
\end{proof}

\begin{corollary}\label{desequiv2} All Desarguesian $(t-1)$-spreads in $\Pq$ are $\PGL(rt,q)$-equivalent.
\end{corollary}

We have seen that to study the stabiliser of a Desarguesian spread, it is sufficient to consider only one particular Desarguesian spread, so we will restrict ourselves to study of the stabiliser of the Desarguesian spread $\D$ obtained from $\Ff: \Fqtr\to \F_q^{rt}$ in Lemma \ref{lem6}.

\begin{theorem} \label{stabiliserspread} Let $\D$ be the Desarguesian $(t-1)$-spread in $\Pq$ obtained from $\Ff: \Fqtr\to \F_q^{rt}$. Suppose that $\la \phi\ra_q$ is an element of $\PGammaL(rt,q)$ stabilising $\D$ elementwise. Then $\phi=\Ff(m_\beta)$, for some $\beta\in \F_{q^t}^*$, where $m_\beta$ is the element of $\GammaL(r,q^t)$ mapping a vector $v$ onto  the vector $\beta v$. Let $\la \phi\ra_q$ be an element of $\PGammaL(rt,q)$ stabilising $\D$, then 
$\phi=\Ff(\xi)$ for some $\xi$ in $\GammaL(r,q^t)$. If $\la\phi\ra_q$ is an element of $\PGL(rt,q)$ stabilising $\D$, then the associated field automorphism $\psi$ fixes every element of $\F_q$.
\end{theorem}
\begin{proof} Suppose that $\la \phi \ra_q (\pi_1)=\pi_2$ for elements $\pi_1,\pi_2\in \D$. Since $\pi_i$ is an $\F_q$-subspace of $\Pq$, we have that $\pi_i=\la U_i\ra_q$ for some $U_i\in \F_{q}^{rt}$, $i=1,2$. This implies that $\la \phi \ra_q (\la U_1\ra_q)=\la \phi(U_1)\ra_q=\la U_2\ra_q$, and hence, that $\phi(U_1)=\lambda U_2=U_2$ for some $\lambda \in \F_q$. Hence, $\phi$ is an element of the stabiliser of the Desarguesian vector space partition and by Theorems \ref{vectorelementwise} and \ref{stabiliservector}, the theorem follows.
\end{proof}

\begin{lemma} \label{transitive}Let $\D$ be the Desarguesian $(t-1)$-spread in $\Pq$ obtained from $\Ff: \Fqtr\to \F_q^{rt}$. The elementwise stabiliser of $\D$ in $\PGammaL(rt,q)$ acts sharply transitively on the points of an element of $\D$.
\end{lemma}
\begin{proof} Let $P$ and $Q$ be two points of the same spread element of $\D$. Then $P=\la \Ff(\alpha_1 v)\ra_q$ and $Q=\la \Ff(\alpha_2 v)\ra_q$ for some $\alpha_1,\alpha_2\in \F_{q^t}^\ast$ and $v\in \F_{q^t}^{r*}$. Consider the element $m_{\frac{\alpha_2}{\alpha_1}}$ of $\GL(r,q^t)$, mapping a vector $w$ onto $\frac{\alpha_2}{\alpha_1}w$. 
Now $\la \Ff(m_{\frac{\alpha_2}{\alpha_1}})\ra_q\la \Ff(\alpha_1 v)\ra_q=\la \Ff(\frac{\alpha_2}{\alpha_1}\alpha_1 v)\ra_q=\la \Ff(\alpha_2 v)\ra_q$ and $\la \Ff(m_{\frac{\alpha_2}{\alpha_1}})\ra_q$ stabilises $\D$ elementwise by Lemma \ref{omgekeerd2}. This shows that the elementwise stabiliser acts transitively on the points of a spread element. Now every element of the elementwise stabiliser in $\PGammaL(rt,q)$ of $\D$ is of the form $\la \Ff(m_\beta)\ra_q$ for some $\beta\in \F_{q^t}^*$, so the stabiliser of $\D$ contains $\frac{q^t-1}{q-1}$ different elements, which is the number of points in one spread element, showing that the elementwise stabiliser acts sharply transitively on the points of a spread element of $\D$.
\end{proof}

For $\PGL(rt,q)$, the following  theorem was already proven in \cite{Dye}, by using different methods.

\begin{theorem}\label{Dye} Let $m_\beta$ be the element of $\GammaL(r,q^t)$ mapping a vector $v$ of $\F_{q^t}$ onto $\beta v$. Let $Z=\{\langle m_\l\rangle_q\mid\l\in \F_q^*\}$. The stabiliser in $\PGammaL(rt,q)$ of a Desarguesian spread in $\Pq$ is isomorphic to $\GL(r,q^t)/Z\rtimes \Aut(\F_{q^t})$. The stabiliser in $\PGL(rt,q)$ of a Desarguesian $(t-1)$-spread in $\Pq$ is isomorphic to $\GL(r,q^t)/Z\rtimes \Aut(\F_{q^t}/\F_q)$.
\end{theorem}
\begin{proof} Let $\D'$ be a Desarguesian $(t-1)$-spread in $\Pq$. By Corollary \ref{desequiv2}, there is an element $\psi\in \PGL(rt,q)$ such that $\psi(\D)=\D'$, with $\D$ the spread obtained from $\Ff$ in Lemma \ref{lem6}. This implies that for every element $\psi' \in \PGammaL(rt,q)$ stabilising $\D'$, $\psi^{-1}\psi'\psi$ stabilises $\D$. Moreover, $\psi^{-1}\psi'\psi$ is an element of $\PGL(rt,q)$ if and only if $\psi'\in \PGL(rt,q)$. This implies that the stabiliser in $\PGL(rt,q)$ (resp. in $\PGammaL(rt,q)$) of $\D'$ is isomorphic to the stabiliser of $\D$ in $\PGL(rt,q)$ (resp. in $\PGammaL(rt,q)$).
The group $Z$, i.e. the group of $\F_q$-scalar maps is a normal subgroup of $\GL(r,q^t)$, so we may consider the quotient group $\GL(r,q^t)/Z=\{m_\l \xi'\mid\l\in \F_q^*,\xi'\in \GL(r,q^t)\}$. The group $\GL(r,q^t)/Z$ can be embedded in $\PGL(rt,q)$ by mapping an element $\la \xi\ra_q$ onto $\la \Ff(\xi)\ra_q$ (note that $\Ff$ provides an embedding of $\GL(r,q^t)$ in $\GL(rt,q)$ and $\PGL(rt,q)\cong \GL(rt,q)/Z$). 
By Theorem \ref{stabiliserspread}, every element $\langle\phi\rangle_q$ of the stabiliser of a Desarguesian spread in $\PGammaL(rt,q)$ is of the form $\langle \phi\rangle_q=\langle\Ff(\xi)\rangle_q$, with $\xi\in \GammaL(r,q^t)$. 

Suppose that $\xi$ has associated field automorphism $\psi$. Since $\Ff$ is an embedding of $\GammaL(r,q^t)$ into $\GammaL(rt,q)$, $\xi$ is the unique element of $\GammaL(rt,q)$ such that $\phi=\Ff(\xi)$. Now $(m_\beta \xi) (\alpha v)=m_\beta(\psi(\alpha)\xi(v))=\beta\psi(\alpha)\xi(v)=\psi(\alpha)(m_\beta \xi)(v)$, hence, $m_\beta \xi$ also has $\psi$ as associated field automorphism, implying that the map $f$ which maps an element $\langle\phi\rangle_q=\la\Ff(\xi)\ra_q$ of the stabiliser of $\D$ in $\PGammaL(rt,q)$ onto the associated field autormorphism $\psi$ of $\xi$ is well-defined.
 By Lemma \ref{omgekeerd2}, we have that every element of $\Aut(\F_q)$ is the image under $f$ of some element $\la \phi\ra_q$ from the stabiliser of $\D$ in $\PGammaL(rt,q)$. The kernel of this map are all elements $\la \Ff(\xi)\ra_q$ with trivial field automorphism hence, this is exactly $\GL(r,q^t)/Z$. When considering the restriction of $f$ to the elements of the stabiliser of $\D$ in $\PGL(rt,q)$, one sees that the kernel is again $\GL(r,q^t)/Z$, but the image is now $\Aut(\F_{q^t}/\F_q)$.
%
\end{proof}

\section{Singer cycles and Desarguesian spreads}



A {\em Singer cycle} of $\P_q^n$ is an element $\sigma$ of $\PGL(n,q)$ such that the group $\langle \sigma\rangle$ acts sharply transitively on the points of $\P_q^n$; the group $\langle \sigma \rangle$ is then called a {\em Singer group}.
The aim of this section is to give a complete proof using the framework developed in this paper for the following theorem of Drudge \cite{Drudge}, relating Singer groups of $\PGL(n,q)$ with Desarguesian spreads.
\begin{theorem} \label{tdr}A Singer cycle $\sigma$ of $\PG(n-1,q)$ has an orbit which is a $(d-1)$-spread if and only if $d\mid n$. In this case there is exactly one such orbit $\S$; the $G$-stabiliser, where $G=\la \sigma\ra$ of any $S\in \S$ is $Stab_G(S)=\langle \sigma^{\theta_n/\theta_d}\rangle$. $\S$ is a geometric $(d-1)$-spread and the factor group $G/Stab_G(S)$ is a Singer group of the corresponding $\PG(n/d-1,q^d)$. Here, $\theta_n=\frac{q^n-1}{q-1}$.
\end{theorem}

The existence of Singer groups follows from the following observation.

\begin{lemma} \label{lem8} Let $\Ff_0:\F_{q^{rt}}\to \F_{q^t}^r$ be a vfr-map and let $m_\alpha:\F_{q^{rt}}\to \F_{q^{rt}}: v\mapsto \alpha v$, $\alpha \in \F_{q^{rt}}^\ast$. Then $\la \Ff_0(m_\alpha)\ra_{q^t}$ is an element of $\PGL(r,q^t)$. 

The group $G=\{ \la \Ff_0(m_\alpha)\ra_{q^t}\mid\alpha\in \F_{q^{rt}}^\ast\}$ is a Singer group (acting on the points of $\Pqt$). 
\end{lemma}
\begin{proof} Note that for every element $\alpha \in \F_{q^{rt}}^\ast$, $\la \Ff_0(m_\alpha)\ra_{q^t}$ is the element of $\PGL(r,q^t)$, which maps $\la \Ff_0(v)\ra_{q^t}$ onto $\la \Ff_0(\alpha v)\ra_{q^t}$ and that $\Ff_0$ is an $\F_{q^t}$-linear map. Now if $\alpha'=\l\alpha$, with $\l\in \F_{q^t}$, then $\la \Ff_0(\alpha'v)\ra_{q^t}=\la \Ff_0(\l\alpha v)\ra_{q^t}=\la \l\Ff_0(\alpha v)\ra_{q^t}=\la \Ff_0(\alpha v)\ra_{q^t}$, hence, there are (at most) $\frac{q^{rt}-1}{q^t-1}$ distinct elements in $G$. Now let $\la v_1\ra_{q^t}$ and $\la v_2\ra_{q^t}$ be two points of $\Pqt$, then $v_1=\Ff_0(u_1)$ and $v_2=\Ff_0(u_2)$ for some $u_1,u_2\in\F_{q^{rt}}^\ast$, hence, $u_2=\beta u_1$ for some $\beta \in \F_{q^{rt}}^\ast$. This implies that $\la \Ff_0(m_\beta)\ra_{q^t}$ maps $\la v_1\ra_{q^t}$ onto $\la v_2\ra_{q^t}$, and hence, that $G$ acts transitively on the points of $\Pqt$, which in turn implies that $| G|=\frac{q^{rt}-1}{q^t-1}$ and that $G$ is a Singer group.
\end{proof}
\begin{corollary} Let $\omega\in \F_{q^{rt}}$ be a generator of $\F_{q^{rt}}^\ast$, then $\sigma=\la \Ff_0(m_\omega)\ra_{q^t}$ is a Singer cycle of $\Pqt$.
\end{corollary}


The following lemma is an easy corollary of the fact that all Singer subgroups are conjugate (see e.g. \cite{Huppert} for $\GL$ or \cite[Section 4.2]{Hirschfeld} for $\PGL$).
\begin{lemma}\label{toegevoegd} Let $\Ff_0:\F_{q^{rt}}\to \F_{q^t}^r$ be a vfr-map. Let $g$ be a Singer cycle of $\Pqt$, then $g=\xi \sigma' \xi^{-1}$, for some $\xi \in \PGL(r,q^t)$ and $\sigma'= \la \Ff_0(m_{\omega'})\ra_{q^t}$, where $\omega'$ is some generator of $\F_{q^{rt}}^\ast$
\end{lemma}
\begin{proof} We have seen that if $\omega\in \F_{q^{rt}}$ is a generator of $\F_{q^{rt}}^\ast$, then $\sigma=\la \Ff_0(m_\omega)\ra_{q^t}$ is a Singer cycle of $\Pqt$. Since all Singer subgroups are conjugate in $\PGL(n,q)$, we have that $g=\xi \sigma^m \xi^{-1}$ for some $m$. Now $\sigma^m=\la \Ff_0(m_{\omega^m})\ra_{q^t}$, and as the order of $\sigma^m$ in $\PGL(r,q^t)$ is the order of $g$ in $\PGL(r,q^t)$, $\omega^m$ has to have the same order as $\omega$ in $\F_{q^{rt}}^\ast$, and hence, $\omega^m$ is a generator $\omega'$. 
\end{proof}

We will need the following lemma which characterises cosets of a subfield in $\F_{q^n}$.
\begin{lemma} \label{orbituniek}Let $S$ be a set of elements of $\F_{q^n}$ satisfying the following conditions:
\begin{itemize}
 \item[(i)] $S$ is closed under taking $\F_q$-linear combinations;
\item[(ii)] for all $\beta\in \F_{q^n}$ either $\beta S=S$ or $\beta S\cap S=\{0\}$;
\item[(iii)] $|S|=q^{d}$, $d\mid n$.
\end{itemize}
Then $S=\alpha V$ with $V=\F_{q^{d}}$ for some $\alpha\in \F_{q^n}^\ast$.
\end{lemma}
\begin{proof} Let $S$ be a set of elements of $\F_{q^n}$ satisfying the three required properties. Put $S=\{s_1,\ldots,s_{q^{d}}\}$, where $s_1\neq 0$, and $V=\{1,v_2,v_3,\ldots,v_d\}$, where $v_i=s_i/s_1$. It is readily checked that the set $V$ satisfies properties $(i)-(ii)-(iii)$. Moreover, since $1\in V$, and $v_i\in V$, the set $v_iV\cap V$ contains $v_i$, and hence, $v_iV=V$, which implies that for all $i$ and $j$, $v_iv_j$ is contained in $V$. It follows that $V$ is closed under multiplication, and since $V$ is also closed under addition by property (i), $V$ is a subfield.
\end{proof}

It is clear from Lemma \ref{lem8} that if we consider the field reduction map $\Ff\Ff_0:\F_{q^{rt}}\to \F_{q}^{rt}$ and as before, $m_\alpha:\F_{q^{rt}}\to \F_{q^{rt}}: v\mapsto \alpha v$, then $G=\{ \la (\Ff \Ff_0)(m_\alpha)\ra_{q}\mid\alpha\in \F_{q^{rt}}^*\}$ is a Singer group (acting on the points of $\Pq$). In the following lemma, we will prove Theorem \ref{tdr} for the Singer group $G$, and later extend it to all Singer groups.

\begin{lemma} \label{orbit1} Let $d\mid n$. Consider vector field reduction maps $ \Ff_0:\F_{q^n}\to \F_{q^d}^{n/d}$ and $\Ff:\F_{q^d}^{n/d}\to \F_q^n$. Let $\sigma=\la (\Ff \Ff_0)( m_\omega) \ra_{q}$ be a Singer cycle of $\P_q^n$, where $\omega$ is a generator of $\F_{q^n}^\ast$. Consider the $(d-1)$-dimensional subspace $\pi=\{\la (\Ff \Ff_0)(\beta v)\ra_q\mid \beta\in \F_{q^d}^\ast\}$ in $\P_q^n$, where $v$ is fixed in $\F_{q^n}^\ast$. The orbit of $\pi$ under the Singer group $\la \sigma\ra$ is a $(d-1)$-spread  $\S$ of $\P_q^n$. Moreover, there is exactly one such orbit and the spread $\S$ is Desarguesian.
\end{lemma}
\begin{proof} As $\la\sigma\ra$ acts transitively on the points of $\P_q^n$, it is clear that the orbit of $\pi$ under $\la \sigma\ra$ covers every point of $\P_q^n$. 

Now assume that there is a point of $\P_q^n$ contained in $\sigma^i(\pi)\cap \sigma^j(\pi)$, then there is a point contained in $\pi\cap \sigma^{j-i}(\pi)$, or equivalently $\la(\Ff \Ff_0)(\beta v)\ra_{q}=\sigma^{j-i}\la (\Ff \Ff_0)(\beta' v)\ra_q$ for some $\beta,\beta'\in \F_{q^d}^*$, i.e. $\beta v=\l \omega^{j-i}\beta' v$ for some $\l\in \F_q^*$. This implies that $\omega^{j-i}\in \F_{q^d}$ and that $\pi=\sigma^{j-i}(\pi)$. Moreover, this also implies that the orbit of $\pi$ is indeed a $(d-1)$-spread.
We now show that there is exactly one such orbit. Consider a $(d-1)$-space $\mu$ of $\P_q^n$ such that the orbit of $\mu$ under $\sigma$ is a $(d-1)$-spread. Then it is not hard to see that the points of $\mu$ correspond to a set of $q^d$ elements of $\F_{q^n}$ satisfying the properties of Lemma \ref{orbituniek}. This implies that the vectors of $\F_{q^n}$ corresponding to the points of $\mu$ are an $\F_{q^n}$-multiple of the vectors corresponding to $\pi$, which exactly means that $\mu$ is in the orbit of $\pi$ under $\sigma$, and hence, that there is exactly one orbit of $\sigma$ that is a $(d-1)$-spread.

Finally, note that by definition, the set of $(d-1)$-spaces $\{\{\la \Ff(\alpha v)\ra_q\mid\alpha \in \F_{q^d}^\ast\}, v\in \F_{q^d}^{n/d*}\}$ is a Desarguesian spread $\D$. Now $\{v\mid v\in \F_{q^d}^{n/d*}\}=\{\Ff_0(v')\mid v'\in \F_{q^n}^*\}$ and hence, $\D$ equals the set $\{\{\la \Ff(\alpha \Ff_0(v'))\ra_q\mid \alpha \in \F_{q^d}^*\}, v'\in \F_q^{n*}\}$. Since $\Ff_0$ is $\F_{q^d}$-linear, this equals $\{\la (\Ff \Ff_0)(\alpha v')\ra_q\mid \alpha \in \F_{q^d}^*\}, v'\in \F_{q^n}^*\}$. From this, we see that the elements $\sigma^j(\pi)$ are contained in $\D$, and hence, that $\S=\D$ is a Desarguesian spread.
\end{proof}
\begin{theorem}\label{drud} Let $d\mid n$, then a Singer cycle of $\P_q^n=\PG(n-1,q)$ has a unique orbit which is a $(d-1)$-spread $\S$. This spread is necessarily Desarguesian. The $G$-stabiliser, where $G=\la g\ra$ of any $S\in \S$, where $\S$ is the unique $(d-1)$-spread obtained as an orbit of a $(d-1)$-space under $G$, is $G_S=\langle g^{\frac{q^n-1}{q^d-1}}\rangle$. The factor group $G/G_S$ is a Singer group of the corresponding projective space $\P_{q^d}^{n/d}=\PG(n/d-1,q^d)$.
\end{theorem}
\begin{proof} By Lemma \ref{toegevoegd}, a Singer cycle $g$ of $\P_q^n$ can be written as $\xi\sigma\xi^{-1}$, where $\sigma=\la(\Ff \Ff_0)(m_\omega)\ra_q$ for some generator $\omega$ of $\F_{q^n}^\ast$ and $\xi\in \PGL(n,q)$. By Lemma \ref{orbit1}, there is a $(d-1)$-space $\pi$ such that the orbit of $\pi$ under $\sigma$ is a Desarguesian $(d-1)$-spread. Since $\xi \sigma^m=g^m\xi$, this implies that the orbit of $\xi(\pi)$ under $g$ is a $(d-1)$-spread, which is Desarguesian because is the image of a Desarguesian spread under the collineation $\xi$. If $\mu$ is a $(d-1)$-space such that the orbit of $\mu$ under $g$ is a $(d-1)$-spread, then it follows that the orbit of $\xi^{-1}(\mu)$ under $\sigma$ is a $(d-1)$-spread, and hence, by the unicity of this orbit, proven in Lemma \ref{orbit1}, that $\xi^{-1}(\mu)=\sigma^j(\pi)$ for some $j$. This implies that $\mu=\xi\sigma^j(\pi)=g^j\xi(\pi)$, and hence, the orbit of $\mu$ under $g$ is the same as the orbit of $\xi(\pi)$ under $g$. 

As $|G_S|.|S^G|=|G|$, and $S^G$ is the spread $\S$, we have that $|G_S|=\frac{q^n-1}{q-1}/\frac{q^n-1}{q^d-1}=\frac{q^d-1}{q-1}$. Note that $\omega^{\frac{q^n-1}{q^d-1}}\in \F_{q^d}^*$. Since $g=\xi\sigma\xi^{-1}$, $|\langle g^{\frac{q^n-1}{q^d-1}}\rangle|=\frac{q^d-1}{q-1}$ and this group stabilises $\xi(\pi)$, where $\pi=\{\la (\Ff \Ff_0)(\beta v)\ra_q\mid \beta\in \F_{q^d}^\ast\}$, and clearly also all other elements of $S$, which are each contained in the orbit of $\xi(\pi)$ under $g$. This proves $G_S=\la g^{\frac{q^n-1}{q^d-1}}\ra$.

As $G\cong \la \la (\Ff \Ff_0)( m_\omega) \ra_{q}\ra$, $G_S\cong \la \la (\Ff \Ff_0)(m_\omega^{\frac{q^n-1}{q^d-1}})\ra_q\ra\cong \la \la (\Ff \Ff_0)( m_\alpha) \ra_{q}\ra $, where $\alpha\in \F_q^{d*}$, we have that $G/G_S\cong \la \la (\Ff \Ff_0)( m_\omega) \ra_{q^d}\ra$, and this last group is precisely a Singer group of $\P_{q^d}^{n/d}=\PG(n/d-1,q^d)$.
\end{proof}



%

\section{Subspreads of a Desarguesian spread}
A $(t'-1)$-subspread of a Desarguesian $(t-1)$-spread $\D$ in $\P_q^{rt}$ is a $(t'-1)$-spread which partitions the elements of $\D$. It is clear that a $(t'-1)$-subspread of $\D$ can only exist if $t'\mid t$. 
It is our goal to show that for every $t'\mid t$, a Desarguesian $(t-1)$-spread $\D$ has a unique Desarguesian $(t'-1)$-subspread. A proof of the uniqueness of Desarguesian $1$-subspreads of a Desarguesian $3$-spread in $\PG(7,q)$ was given in \cite[Theorem 2.4]{BCQ}. The general theorem was later proven in \cite{ABB}. Both proofs use the theory of indicator sets.


\begin{lemma} \label{deelruimtespread} Let $\Ff_1:\F_{q^{t}}^{r} \to \F_{q^{t'}}^{rt/{t'}}$ and $\Ff_2:\F_{q^{t'}}^{rt/{t'}}\to \F_q^{rt}$ be vfr-maps. Let $T$ be a $(t-1)$-space of $\Pq$ partitioned by elements of the form $\{\la \Ff_2\Ff_1(\beta v)\ra_q\mid \beta\in \F_{q^{t'}}^*\}$, where $v\in \F_{q^t}^{r*}$, then $T=\{\la \Ff_2(U)\ra_q\}$ where $U$ is a $t/t'$-dimensional vector subspace of $\F_{q^{t'}}^{rt/{t'}}$. \end{lemma}
\begin{proof} Consider the set $U$ of all vectors $\Ff_1(\beta v)$ in $\F_{q^{t'}}^{rt/{t'}}$, with $\beta\in \F_{q^{t'}}^*$ and $v\in \F_{q^t}^{r*}$, such that $\la \Ff_2(\Ff_1(\beta v))\ra_q$ is a point of $T$. We will show that $U\cup \{0\}$ is a vector subspace of  $\F_{q^{t'}}^{rt/{t'}}$. Take $\Ff_1(\beta v)$  and $\Ff_1(\beta' v')$ in $U$, then $\Ff_1(\beta v)+\beta''\Ff_1(\beta' v')$, where $\beta''\in \F_{q^{t'}}^*$, equals $\Ff_1(\beta v+\beta''\beta'v')$ since $\Ff_1$ is an $\F_{q^{t'}}$-linear map. We have that $\la\Ff_2\Ff_1(\beta'v')\ra_q$ is a point of $T$, and hence, since $T$ is partitioned by elements of the form $\{\la \Ff_2\Ff_1(\beta v)\ra_q\mid \beta\in \F_{q^{t'}}\}$, and $\beta''\in \F_{q^{t'}}^*$, also $\la\Ff_2\Ff_1(\beta'\beta''v')\ra_q$ is a point of $T$. Since $T$ is a subspace of $\Pq$, this implies that $\la\Ff_2\Ff_1(\beta v)+\Ff_2\Ff_1(\beta'\beta''v')\ra_q=\la \Ff_2(\Ff_1(\beta v)+\Ff_1(\beta'\beta''v'))\ra_q$ is a point of $T$, and hence, that $\Ff_1(\beta v)+\beta''\Ff_1(\beta'v')$ is in $U$.
\end{proof}

The following theorem already appeared in \cite{ABB}, but, as said before, it was proven using the theory of indicator sets.
\begin{theorem} \label{uniek} Let $\D$ be a Desarguesian $(t-1)$-spread in $\P_q^{rt}=\PG(rt-1,q)$. For every $t'\mid t$, there is a unique Desarguesian $(t'-1)$-subspread in $\D$.
\end{theorem}
\begin{proof} Let $\Ff_1:\F_{q^{t}}^{r} \to \F_{q^{t'}}^{rt/{t'}}$, let $\Ff_2:\F_{q^{t'}}^{rt/{t'}}\to \F_q^{rt}$, and let $\Ff=\Ff_2\Ff_1$ be vector field reduction maps. Since all Desarguesian spreads are $\PGL$-equivalent (Corollary \ref{desequiv2}), we may choose $\D$ to be the Desarguesian spread with elements $\{\{\la \Ff(\alpha v)\ra_q \mid \alpha \in \F_{q^t}^\ast\},v\in \F_{q^t}^{r\ast}\}$. Suppose that there is a Desarguesian $(t'-1)$-subspread $\D_1$ of $\D$ and let $\D_2$ be the Desarguesian $(t'-1)$-spread with elements $\{\la \Ff(\beta v)\ra_q \mid \beta \in \F_{q^{t'}}^\ast\}$, for all $v\in \F_{q^t}^{r*}$. It is clear that $\D_2$ is $(t'-1)$-subspread of $\D$. Again using that all Desarguesian spreads are $\PGL$-equivalent, we find that there is an element $ \xi$ of $\PGL(rt,q)$ such that $\xi(\D_1)=\D_2$. 

By definition,  every element of $\D$ is of the form $\{\la \Ff_2(\Ff_1(\alpha v)\ra_q\mid \alpha\in \F_{q^{t}}^*\}$. Moreover, $\{\{\la \Ff_1(\alpha v)\ra_{q^{t'}}\mid \alpha\in \F_{q^{t}}^*\},v\in \F_{q^t}^{r*}\}$ forms a Desarguesian spread $\mathcal{E}$ in $\P_{q^{t'}}^{rt/t'}$. Let $\la U\ra_{q^{t'}}\in \mathcal{E}$ and $\beta \in \F_{q^{t'}}$, then $U=\{\Ff_1(\alpha v)\mid \alpha \in \F_{q^t}\}$ and $\beta U=\{\beta\Ff_1(\alpha v)\mid \alpha\in \F_{q^t}\}=\{\Ff_1(\alpha' v)\mid \alpha'\in \F_{q^t}\}=U$ since $\Ff_1$ is $\F_{q^{t'}}$-linear.
Since every element of $\D$ is partitioned by elements of $\D_1$, every element of $\xi(\D)$ is partioned by elements of $\xi(\D_1)=\D_2$. By Lemma \ref{deelruimtespread}, we find that every element $F_i$, $i=1,\ldots,\frac{q^{rt}-1}{q^t-1}$ of $\xi(\D)$ is of the form $\{\la \Ff_2(U'_i)\ra_q\}$ for some vector subspace $U'_i$ of $\F_{q^t}^r$.  It is clear that the vector subspaces $U'_i$, $i=1,\ldots, \frac{q^{rt}-1}{q^t-1}$ form a vector space partition of $\F_{q^{t'}}^{rt/t'}$, and hence, determine a $(t/t'-1)$-spread $\mathcal{E'}$ of $\P_{q^{t'}}^{rt/t'}$. We have that if  $\la U'\ra_{q^{t'}}\in \mathcal{E'}$, then $\la \Ff_2(U')\ra_q\in \xi(\D)$. Moreover, the incidence structure with points the elements $\la U'_i\ra_{q^{t'}}$ and lines the spaces spanned by two of such elements is isomorphic to the incidence structure with points the elements of the Desarguesian spread $\xi(\D)$ and lines the spaces spanned by two of such elements, which implies that the spread $\E'$ is Desarguesian. Again using that all Desarguesian spreads of $\P_{q^{t'}}^{rt/t'}$ are $\PGL$-equivalent, yields that there is an element $\la \psi\ra_{q^{t'}}$, where $\psi \in \GL(rt/t',t')$ such that $\la \psi \ra_{q^{t'}}(\mathcal{E})=\mathcal{E'}$.

Let $D\in \D$, then $D=\la \Ff_2(U)\ra_q$ for some $\la U\ra_{q^{t'}}\in \mathcal{E}$. Consider the mapping $\xi':\la \Ff_2(\psi)\ra_q$, then $\xi'(D)=\la \Ff_2(\psi)\ra_q(\la \Ff_2(U)\ra_q)=\la \Ff_2(\psi(U))\ra_q$. Now, since $\la U\ra_{q^{t'}}\in \mathcal{E}$, and $\la \psi\ra_{q^{t'}}(\mathcal{E})=\mathcal{E'}$, $\la \psi\ra_{q^{t'}}(\la U\ra_{q^{t'}})=\la \psi(U)\ra_{q^{t'}}$ is contained in $\mathcal{E'}$. But this implies that $\la \psi(U)\ra_{q^{t'}}=\la U'\ra_{q^{t'}}$ for some $U'$ with $\la \Ff_2(U')\ra_q\in \xi(\D)$ for all $D\in \D$. Hence, $\xi'(D)$ is contained in $\xi(\D)$ for all $D\in \D$. We conclude that $\xi'(\D)=\xi(\D)$. This implies that $\xi^{-1} \xi'$ is an element of the stabiliser of $\D$ in $\PGammaL(rt,q)$, and hence by Theorem \ref{stabiliserspread}, is of the form $\la\Ff(\chi)\ra_q$, where $\chi\in \GammaL(r,q^t)$. This implies that $\xi^{-1}=\la\Ff(\chi)\ra_q \la \Ff_2(\psi)\ra_q^{-1}=\la \Ff_2(\Ff_1(\chi) \psi^{-1})\ra_q$ and hence, is an element of $\PGammaL(rt/t',q^{t'})$. This implies that $\xi^{-1}$ maps an element $(\la \beta v\ra_{q}\mid \beta \in \F_{q^{t'}})$ of $\D_2$ onto an element contained in $\D_2$, and hence, $\D_1=\xi^{-1}(\D_2)=\D_2$.
\end{proof}

\section{Linear sets and Condition (A) of Csajb\'ok-Zanella} \label{cor}
{\em Linear sets} have been intensively used in recent years in order to classify, construct or characterise various structures related to finite incidence geometry. Besides their relation to blocking sets, introduced in \cite{guglielmo}, linear sets also appear in the study of translation ovoids (see e.g. \cite{LP}), and are of high relevance in the theory of semifields (see e.g. \cite{LaPo}).

Roughly speaking, linear sets are point sets in a projective space that are obtained from the inverse image of a vector space under some field reduction map, but we will now give an exact definition. Recall that an $\F_{q^t}$-subspace of the projective space $\Pqt$ is a set $\{\la v\ra_{q^t}\mid v \in U^*\}$, where $U$ is a vector subspace of $\Fqtr$. Now if $\Ff:\Fqtr\rightarrow\F_q^{rt}$ is a vfr-map, then, we can consider the set
$$\B(U)=\{\la v\ra_{q^t}\mid v \in \Ff^{-1}(U^*)\}, \textrm{where}\ U\ \textrm{is a vector subspace of}\ \F_q^{rt}.$$
The set $\B(U)$ consists of points of $\Pqt$ and forms an $\F_{q^t}$-subspace of $\Pqt$ if and only if $\Ff^{-1}(U)$ is a vector subspace of $\Fqtr$ (over $\F_{q^t}$), which is certainly not always the case. But for $v,w\in \Ff^{-1}(U)$, it does hold that $v+\l w$, with $\l\in \F_q$ is an element of $\Ff^{-1}(U)$. For this reason, $\B(U)$ is called an {\em $\F_q$-linear set.}

\begin{definition}
Let $\Ff:\Fqtr\rightarrow\F_q^{rt}$ be a vfr-map and $\pi$ be a subspace of $\Pq$, say $\pi=\la U\ra_q$ with $U$ a vector subspace of $\F_q^{rt}$. Then we define
$$\B(\pi)=\B(U) (=\{\la v\ra_{q^t}\mid v \in \Ff^{-1}(U^*)\}).$$
\end{definition}
We see that, since  $\l U=U$ if $\l \in \F_q$, $\B(\pi)$ is well defined. 
%

For the remainder of this section, we will fix the vector field reduction map $\Ff:\Fqtr\rightarrow \F_q^{tr}$, and the Desarguesian spread $\D$ obtained from $\Ff$. In this way, the notation $\B(U)$ is well-defined.

In \cite{corrado}, Csajb\'ok and Zanella consider the following Condition (A) for $(L,n)$, where $L$ is a linear set:
\begin{itemize}
\item[(A)]For any two $(n-1)$-subspaces $\pi,\pi'\subset \PG(rt-1,q)$, such that $\B(\pi)=L=\B(\pi')$, a collineation $\gamma\in \PGammaL(rt,q)$ exists such that 
\begin{itemize}
\item[(a)]$\gamma(\pi)=\pi'$
\item[(b)] $\gamma$ preserves the Desarguesian spead $\D$
\item[(c)] the induced map on $\PG(r,q^t)$ is a collineation.
\end{itemize}
\end{itemize}
In \cite{corrado}, the authors show that Condition (A) has to hold for $(L,n)$ in order for the equivalence statement for linear sets of \cite{michel1} to be true for $L$. They provide an example of a linear set not satisfying Condition (A):
\begin{theorem}\cite{corrado}
The scattered linear set $L$ of pseudoregulus type, $L=\{(\lambda,\lambda^q)\mid \lambda \in \F_{q^5}\}$ does not satisfy Condition (A).
\end{theorem}
The goal of this section is to have a better understanding of Condition (A) by providing equivalent statements and by giving more information about examples of linear sets that are known to (not) satisfy Condition (A).

We start with an easy observation.

\begin{lemma} \label{c}Property $(c)$ in Condition (A) follows from $(b)$ (and is thus superfluous).
\end{lemma}
\begin{proof} From Theorem \ref{stabiliserspread}, we obtain that an element $\gamma$ which satisfies (b) is of the form $\la \Ff(\xi)\ra_q$ where $\xi\in \GammaL(r,q^t)$. This implies that the induced map of $\gamma$ on $\Pqt$ is $\la \xi\ra_{q^t}$, which is an element of $\PGammaL(r,q^t)$ and hence, a collineation.
\end{proof}

We can now rewrite condition (A) as follows.
\begin{theorem} \label{AenA'} Condition (A) is equivalent to the following.
\begin{itemize}
\item[(A')] For any two $U,U'$ vector subspaces of dimension $n$ of $\F_q^{rt}$ with $\B(U)=\B(U')$, there is an element $\xi\in \GammaL(r,q^t)$ such that $\Ff(\xi)(U)= U'$.
\end{itemize}
\end{theorem}
\begin{proof} By Lemma \ref{c}, Property (c) does not need to be checked. We have by Theorem \ref{stabiliserspread} that Condition (A) is equivalent to the following condition: for any two $(n-1)$-subspaces $\pi,\pi'\subset \Pq$ with $\B(\pi)=\B(\pi')$, there is an element $\la \Ff(\xi')\ra_{q}$, with $\xi\in \GammaL(r,q^t)$ such that $\la \Ff(\xi')\ra_q (\pi)=\pi'$.

Now any two $(n-1)$-subspaces $\pi,\pi'\subset \Pq$ are of the form $\la U\ra_q, \la U'\ra_q$, for some $U,U'$ $n$-dimensional vector subspaces of $\F_q^{rt}$. We see from the definition that $\B(\la U\ra_q)=\B(\la U'\ra_q)$ if and only if $\B(U)=\B(U')$. Moreover, there is an element $\la \Ff(\xi')\ra_{q}$, with $\xi'\in \GammaL(r,q^t)$ such that $\la \Ff(\xi')\ra_q (\pi)=\pi'$ if and only if $\la \Ff(\xi')\ra_{q}(\la U\ra_q)=\la U'\ra_q$, and hence, if and only if there is and element $\xi'\in \GammaL(r,q^t)$ and some $\lambda\in \F_q$ such that $\lambda\Ff(\xi')(U)=U'$, or equivalently, (put $\xi=m_\lambda  \xi'$) if and only if there is an element $\xi\in \GammaL(r,q^t)$ such that $\Ff(\xi)(U)= U'$.
\end{proof}

It is now natural to define the following second condition for a linear set $L$:

\begin{itemize}
\item[(B)] Let $U$ and $U'$ be two $n$-dimensional subspaces of $\F_q^{rt}$ with $\B(U)=L=\B(U')$, then $U'=\Ff(m_\beta)(U)$ for some $\beta\in \F_{q^t}^*$.
\end{itemize}

The following theorem links property (B) with the representation problem for linear sets, as described in \cite{michel1}.
\begin{theorem}\label{rep} Let $\pi$ and $\pi'$ be two $(n-1)$-spaces of $\Pq$ intersecting in a point $P=\la \Ff(v)\ra_q$ with $\pi\cap \{\la \Ff(\alpha v)\ra_q\mid \alpha \in \F_{q^t}^*\}=\pi'\cap\{\la\Ff(\alpha v)\ra_q\mid\alpha \in \F_{q^t}^*\}=\{P\}$. Suppose that $\B(\pi)=\B(\pi')$ and that Condition (B) holds, then $\pi=\pi'$.
\end{theorem}
\begin{proof} Let $\pi$ and $\pi'$ be two $(n-1)$-spaces satifying the properties of the statement. Let $\pi=\la U\ra_q$ and $\pi'=\la U'\ra_q$. Since $\B(\la U\ra_q)=\B(\la U'\ra_q)$, $\B(U)=\B(U')$ and by Condition (B), $U'=\Ff(m_\beta)(U)$ for some $\beta\in \F_{q^t}^*$. Now $\la \Ff(m_\beta)\ra_q(\la U\ra_q)=\la U'\ra_q$ and $\la \Ff(m_\beta)\ra_q$ stabilises $\D$ elementwise by Lemma \ref{omgekeerd2}, which implies that $\la \Ff(m_\beta)\ra_q(\la \Ff(v)\ra_q)=\la \Ff(v)\ra_q$. From this, we see that $\beta\in \F_q^*$, and hence, $\pi'=\la U'\ra_q=\la \Ff(m_\beta)\ra_q(\la U\ra_q)=\la U\ra_q=\pi$.
\end{proof}
Since Condition (A) is equivalent to Condition (A'), we immediately obtain the following.
\begin{theorem} \label{BimpliesA} Condition (B) implies Condition (A).
\end{theorem}

In \cite{corrado}, the authors note that it is shown in \cite[Proposition 2.3]{bonoli} that condition (A) holds for linear blocking sets. This is true, but in fact it is shown in \cite[Proposition 2.3]{bonoli} that the stronger condition (B) also holds. Condition (A) is not equivalent to Condition (B) since it follows from \cite{corrado} that the scattered linear set $L=\{\la (\lambda,\lambda^q)\ra_{q^3}\mid \lambda\in \F_{q^3}\}$ satisfies Condition (A). However, it follows from Theorem \ref{rep} that it does not satisfy Condition (B): through a point $P$ of $\PG(5,q)$, contained in an element of the Desarguesian spread corresponding to $L$, there are exactly two planes $\pi$, $\pi'$, both meeting the spread element through $P$ exactly in the point $P$, such that $\B(\pi)=\B(\pi')=L$ (for a proof, see \cite{michel1}).  

Our goal is to prove an equivalent condition for Condition (A) in Theorem \ref{equiv}.

\begin{lemma}\label{lem3}
Suppose that there is an element $\xi\in \GammaL(r,q^t)$ such that $\Ff(\xi)(U)=U'$. Then $\la \xi\ra_{q^t}(\B(U))=\B(U')$.
\end{lemma}
\begin{proof} Let $\la v\ra_{q^t}$ be a point of $\B(U)$, then $\Ff(v)=u$ for some $u\in U^*$. Now $\Ff(\xi)(U)=U'$, so $\Ff(\xi)(\Ff(v))=\Ff(\xi(v))\in U'^*$. This implies that $\xi(v)$ in $\Ff^{-1}(U'^*)$, and hence, that $\la \xi(v)\ra_{q^t} \in \B(U')$. This shows that $\la \xi\ra_{q^t}\la v\ra_{q^t}\in \B(U')$ for all $\la v\ra_{q^t}\in \B(U)$.
\end{proof}

\begin{lemma}\label{lem4} Suppose that Condition (A) holds and let $\pi$ and $\pi'$ be $(n-1)$-dimensional subspaces of $\Pq$ with $\B(\pi)=\B(\pi')$. The element $\la \xi\ra_{q^t}$ obtained from the equivalent condition (A') is an element of the stabiliser of $\B(\pi)$ in $\PGammaL(r,q^t)$. \end{lemma}
\begin{proof} Let $\pi=\la U\ra_q$ and $\pi'=\la U'\ra_q$, then $\B(U)=\B(U')$. Since Condition (A) is equivalent to Condition (A') by Theorem \ref{AenA'}, we have that there is an element $\xi\in \GammaL(r,q^t)$ such that $\Ff(\xi)(U)=U'$. By Lemma \ref{lem3}, $\la \xi\ra_{q^t}(\B(U))=\B(U')$, and hence $\la \xi\ra_{q^t}(\B(\pi))=\B(\pi')=\B(\pi)$. This implies that $\la \xi\ra_{q^t}$ stabilises $\B(\pi)$.
\end{proof}

Recall that $\bar{\Ff}$ denotes the mapping that maps a point $\la v\ra_{q^t}$ of $\Pqt$ onto the element $\{\la \Ff(\alpha v)\ra_q\mid\alpha \in \F_{q^t}^\ast\}$ of the Desarguesian spread $\D$. Then $\bar{\Ff}(\B(\pi))$ is the set of spread elements of $\D$ corresponding to the points of $\B(\pi)$, and it is easy to see that $\bar{\Ff}(\B(\pi))$ is precisely the set of spread elements meeting $\pi$. 
\begin{lemma} \label{lem5} $|\PGammaL(rt,q)_{\D,\bar{\Ff}(\B(\pi))}|=|\PGammaL(r,q^t)_{\B(\pi)}|(q^t-1)/(q-1)$.
\end{lemma}
\begin{proof} Recall the fact that $\Ff$ provides an embedding of $\GammaL(r,q^t)$ into $\GammaL(rt,q)$ (Theorem \ref{embedding}). Consider the subgroup $G$ of $\GammaL(r,q^t)$ consisting of all elements $\xi$ such that $\la \xi\ra_{q^t}$ stabilises $\B(\pi)$. We see that every element of the subgroup $\{\Ff(g)\mid g\in G\}$ of $\GammaL(rt,q)$ has the property that $\la \Ff(g)\ra_q$ stabilises $\D$ and $\bar{\Ff}(\B(\pi))$. On the other hand, if we have an element $\la \psi\ra_q$ of $\PGammaL(rt,q)$ such that $\la \psi\ra_q$ stabilises $\D$, then we have shown in Theorem \ref{stabiliserspread} that $\psi=\Ff(\xi)$ for some $\xi\in \GammaL(r,q^t)$; if $\la \Ff(\xi)\ra_q$ stabilises $\bar{\Ff}(\B(\pi))$, then $\la\xi\ra_{q^t}(\B(\pi))=\B(\pi)$. 
This implies that $|\PGammaL(r,q^t)_{\B(\pi)}|=|G|/(q^t-1)$ and $|\PGammaL(rt,q)_{\D,\bar{\Ff}(\B(\pi))}|=|\Ff(G)|/(q-1)=|G|/(q-1)$.
\end{proof}

\begin{theorem} \label{equiv}Condition (A) holds for $(L=\B(\pi),n)$, where $\pi$ is an $(n-1)$-space in $\Pq$ if and only if the number of $(n-1)$-spaces $\pi'$ in $\Pq$ with $\B(\pi')=L$ equals
$$\frac{|\PGammaL(r,q^t)_{\B(\pi)}|}{|\PGammaL(rt,q)_{\D,\pi}|}\frac{(q^t-1)}{(q-1)}.$$
\end{theorem}
\begin{proof} Let $X$ be the number of $(n-1)$-spaces $\pi'$ in $\Pq$ with $\B(\pi')=L$. Note that if $\la \xi\ra_{q^t}$ stabilises $\B(\pi)$, then $\la \Ff(\xi)\ra_q$ stabilises $\bar{\Ff}(\B(\pi))\ra_q$. Using Lemma \ref{lem4}, we see that Condition (A) holds if and only if every $\pi'$ for which $\B(\pi')=L$ is in the orbit of $\pi$ under $\PGammaL(rt,q)_{\D,\bar{\Ff}(\B(\pi))}$. Hence, if and only if $X$ equals the length of this orbit, which is $|\PGammaL(rt,q)_{\D,\bar{\Ff}(\B(\pi))}|/|\PGammaL(rt,q)_{\D,\bar{\Ff}(\B(\pi)),\pi}|$. Now since $\bar{\Ff}(\B(\pi))$ is the set of elements of $\D$ meeting $\pi$, it is clear that $\PGammaL(rt,q)_{\D,\bar{\Ff}(\B(\pi)),\pi}$ equals $\PGammaL(rt,q)_{\D,\pi}$ and by Lemma \ref{lem5}, $|\PGammaL(rt,q)_{\D,\bar{\Ff}(\B(\pi))}|=|\PGammaL(r,q^t)_{\B(\pi)}|(q^t-1)/(q-1)$.
\end{proof}
\begin{lemma} If $X$ is the number of $(n-1)$-spaces $\pi$ in $\Pq$ with $\B(\pi)=L$, then the number of $(n-1)$-spaces through a fixed point $Q\in \B(\pi)$ equals $X\frac{q-1}{q^t-1}.|\B(Q)\cap \pi|$.
\end{lemma}
\begin{proof} As the elementwise stabiliser of $\D$ acts transitively on the points of $\B(\pi)$ by Lemma \ref{transitive}, for each point $Q$, the number of $(n-1)$-spaces $\pi$ through $Q\in \B(\pi)$ equals the average number of such spaces, i.e. $X\frac{q-1}{q^t-1}.|\B(Q)\cap \pi|$.
\end{proof}

\begin{corollary} \label{gevA} Condition (A) holds for $(L=\B(\pi),n)$, where $\pi$ is an $(n-1)$-space in $\Pq$ if and only if the number of $(n-1)$-spaces $\pi'$ in $\Pq$ with $\B(\pi')=L$ through a fixed point $Q$ in $\B(\pi)$ equals
$$\frac{|\PGammaL(r,q^t)_{\B(\pi)}|}{|\PGammaL(rt,q)_{\D,\pi}|}|\B(Q)\cap \pi|.$$
\end{corollary}
Finally, we end this section with two corollaries providing equivalent descriptions for Conditions (A) and (B).
\begin{lemma} \label{gevB} Condition (B) holds for $(L=\B(\pi),n)$, where $\pi$ is an  $(n-1)$-space in $\Pq$ if and only if the number of $(n-1)$-spaces $\pi'$ in $\Pq$ with $\B(\pi')=L$ in $\B(\pi)$ equals
$$\frac{q^t-1}{q-1},$$ and if and only if the number of $(n-1)$-spaces $\pi'$ in $\Pq$ with $\B(\pi')=L$ in $\B(\pi)$ through any fixed point $Q$ in $\B(\pi)$ equals
$$|\B(Q)\cap \pi|.$$
\end{lemma}

\begin{corollary} If Condition (B) holds for $(L=\B(\pi),n)$ then $|\PGammaL(r,q^t)_{\B(\pi)}|=|\PGammaL(rt,q)_{\D,\pi}|$. If $|\PGammaL(r,q^t)_{\B(\pi)}|=|\PGammaL(rt,q)_{\D,\pi}|$ then Conditions (A) and (B) are equivalent for $(L=\B(\pi),n)$. \end{corollary}
\begin{proof} By Theorem \ref{BimpliesA}, if Condition (B) holds, then Condition (A) holds, and hence, by Lemma \ref{equiv}, the number of $(n-1)$-spaces $\pi'$ in $\Pq$ with $\B(\pi')=L$, say $X$, equals $\frac{|\PGammaL(r,q^t)_{\B(\pi)}|}{|\PGammaL(rt,q)_{\D,\pi}|}\frac{q^t-1}{q-1}$. But by Corollary \ref{gevB}, Condition (B) holds if and only if $X=\frac{q^t-1}{q-1}$, so this implies that $|\PGammaL(r,q^t)_{\B(\pi)}|=|\PGammaL(rt,q)_{\D,\pi}|$. 

Vice versa, suppose that $|\PGammaL(r,q^t)_{\B(\pi)}|=|\PGammaL(rt,q)_{\D,\pi}|$, and that Condition (A) holds, then by Lemma \ref{equiv}, $X=\frac{|\PGammaL(r,q^t)_{\B(\pi)}|}{|\PGammaL(rt,q)_{\D,\pi}|}\frac{q^t-1}{q-1}=\frac{q^t-1}{q-1}$. Now Corollary \ref{gevB} shows that Condition (B) holds.
\end{proof}

\section{Appendix: a remark on the embeddability of $\PGL(r,q^t)$ in $\PGL(rt,q)$}\label{appendix}
We have seen in Remark \ref{opmerkinginbedding} that the natural extension of the field reduction map $\bar{\Ff}$ to elements of $\PGL(r,q^t)$ does not provide an embedding of $\PGL(r,q^t)$ in $\PGL(rt,q)$. In \cite{guglielmo}, the author claims to have constructed an embedding of a certain group $G$ isomorphic to $\PGL(r,q^t)$ embedded in $\PGL(rt,q)$ but one can check that the assertion that the subgroup $G$ is isomorphic to $\PGL(r,q^t)$ is not correct: the group $G$ is isomorphic to $\GL(r,q^t)/Z$ in $\PGL(rt,q)$, and is exactly the stabiliser of the Desarguesian $(t-1)$-spread in $\PG(rt-1,q)$. 
But these facts of course do not imply that it is in general impossible to embed $\PGL(r,q^t)$ in $\PGL(rt,q)$; we have the following theorem.
\begin{lemma} If $\gcd(q^t-1,r)=1$, then $\PGL(r,q^t)$ can be embedded in $\PGL(rt,q)$.
\end{lemma}
\begin{proof} For this proof, we will switch between the representation of the elements of $\GL$ as matrices, and the representation as linear maps. Denote by $\SL(r,q^t)$ the group of $(r\times r)$-matrices with entries in $\F_{q^t}$ that have determinant one. Let $Z$ denote the center of $\mathrm{GL}(r,q^t)$, so $Z$ is the set of scalar matrices. Recall that $\PSL(r,q^t)=\SL(r,q^t)/(Z\cap \SL(r,q^t))$. Since $Z\cap \SL(r,q^t)$ consists of all scalar matrices with determinant one, it is easy to see that if $\gcd(q^t-1,r)=1$, then this set only contains the identity matrix, and hence in this case, $\PSL(r,q^t)=\SL(r,q^t)$.
Now $\PSL(r,q^t))\leq \PGL(r,q^t)$, and hence, if $\gcd(r,q^t-1)=1$, $\SL(r,q^t)\leq \PGL(r,q^t)$. But $|\SL(r,q^t)|=\frac{|\GL(r,q^t)|}{q^t-1}=|\PGL(r,q^t)|$ and hence, $\SL(n,q^t)$ and $\PGL(r,q^t)$ are isomorphic if $\gcd(r,q^t-1)=1$.

Now we have the following:
$$\SL(r,q^t)\overset{\kappa}{\hookrightarrow} \GL(r,q^t)\overset{\Ff}{\hookrightarrow} \GL(rt,q)\overset{\kappa'}{\twoheadrightarrow} \PGL(rt,q),$$
where $\kappa$ is inclusion and $\kappa'$ is the projection map.
The composition $\iota=\kappa'\Ff\kappa$ maps an element of $\SL(r,q^t)$ onto an element of $\PGL(rt,q)$. The kernel of $\kappa'$ are the $rt\times rt$-matrices scalar matrices over $\F_q$. Now since $1=\gcd(q^t-1,r)=\gcd(q-1,r)$, the only scalar matrix contained in $\SL(r,q^t)$ is the $r\times r$ identity matrix. As $\kappa$ and $\Ff$ are injective maps, we have that $\iota$ has a trivial kernel, and hence, provides an embedding.
\end{proof}

One can see that the condition $\gcd(q^t-1,r)=1$ is not necessary, as the {\em twisted tensor product embedding} embeds $\PGL(2,q^2)$ in $\PG(4,q)$ (in general: $\PGL(r,q^t)$ in $\PGL(r^t,q)$) for all $q$, and hence, also for odd $q$ which has $\gcd(q^2-1,2)=2$. 

However, one should note that we are not interested in an arbitrary embedding of $\PGammaL(r,q^t)$; we are interested in finding those embeddings $\iota$ of $\PGammaL(r,q^t)$ in $\PGammaL(rt,q)$ such that the image of $\iota$ stabilises the Desarguesian spread, i.e. embeddings of $\PGammaL(r,q^t)$ in $\GL(r,q^t)/Z\rtimes \Aut(\F_{q^t})$.
John Sheekey informed us that for $r=2$ and $q^t=3\mod 4$, this is always possible, whereas he showed by computer that for $r=2$ and $q^t=9,25,49$ or $r=2$ it is impossible to find  an embedding of $\PGL(2,q^t)$ in $\GL(2t,q)/Z\rtimes \Aut(\F_{q^t}/\F_q)$ \cite{John}. To the author's knowledge, the general problem of deciding whether there exists an embedding of $\PGL(r,q^t)$ in $\GL(rt,q)/Z\rtimes \Aut(\F_{q^t}/\F_q)$ is an open problem.

\end{document}